\documentclass[11pt,twoside]{amsart}
\usepackage{amssymb,amsmath}
\usepackage{amscd}
\usepackage{amsthm}
\usepackage{latexsym}
\usepackage[noadjust]{cite}
\usepackage{mathrsfs}
\usepackage{indentfirst}
\usepackage{graphicx}
\usepackage{graphicx}
\usepackage{subfigure}
\usepackage{color}
\usepackage{enumerate}
\allowdisplaybreaks[4]

\def\dfrac{\displaystyle\frac}
\let\oldsection\section
\renewcommand\section{\setcounter{equation}{0}\oldsection}

\renewcommand\theequation{\thesection.\arabic{equation}}
\newtheorem{theorem}{Theorem}[section]
\newtheorem{lemma}[theorem]{Lemma}
\newtheorem{proposition}[theorem]{Proposition}
\newtheorem{definition}[theorem]{Definition}
\newtheorem{remark}[theorem]{Remark}

\newtheorem{corollary}[theorem]{Corollary}

\usepackage[colorlinks,
linkcolor=blue,       
anchorcolor=red,  
citecolor=blue,        
]{hyperref}

\def\R{\mathbb{R}}
\def\M{\mathcal{M}}

\title[Normalized solutions for the planar Schr\"{o}dinger-Poisson system]
{Normalized solutions for the planar Schr\"{o}dinger-Poisson system with two electrons interaction}

\author[B. Li]{Baihong Li}
\address{College of Mathematics, Jilin University, Changchun 130012, PR CHINA}
\email{bhlimath@zohomail.com}

\author[Y. Wei]{Yuanhong Wei}
\address{College of Mathematics, Jilin University, Changchun 130012, PR CHINA}
\email{weiyuanhong@jlu.edu.cn}

\author[X. Zeng]{Xiangjian Zeng}
\address{Faculty of Mathematics, Informatics and Mechanics, University of Warsaw, ul. Banacha 2, 02-097 Warsaw, POLAND \& 
	Institute of Mathematics, Polish Academy of Science, ul. Śniadeckich 8, 00-656 Warsaw, POLAND 
}
\email{xzeng@impan.pl}

\begin{document}

	\begin{abstract}
		
		This paper focuses on the normalized solutions for the planar Schr\"{o}dinger-Poisson system with a  two-electron interaction, which models the effect between electrons and the electrostatic potential they generate.  As the parameters vary, some existence results are established. Specifically, a ground state solution is obtained for some general cases. The existence of two solutions is established for the mass-supercritical case, one of which is a  ground state solution and the other one is an excited state solution. 
		We develop  a compactness method to deal with the functionals involving logarithmic convolution terms. The  Poho\v{z}aev identity for the coupled  Schr\"{o}dinger-Poisson system with a logarithmic convolution term is also shown, which is crucial for addressing the mass-supercritical problem. 
		
	\end{abstract}
	\maketitle
	

	\tableofcontents
	\section{Introduction}
	The present paper focuses on the following	  Schr\"{o}dinger-Poisson system with two electrons interaction: 
	\begin{equation}\label{evolutional system}
		\begin{cases}
			\begin{aligned}
				-i\partial_t\Phi_1-\Delta\Phi_1 +w\Phi_1&=\mu_{1}\left|\Phi_1\right|^{2p-2}\Phi_1+\beta\left|\Phi_2\right|^{p}\left|\Phi_1\right|^{p-2}\Phi_1,\quad &&\text{in } \R\times\R^N,\\
				-i\partial_t\Phi_2-\Delta\Phi_2 +w\Phi_2&=\mu_{2}\left|\Phi_2\right|^{2p-2}\Phi_2+\beta\left|\Phi_1\right|^{p}\left|\Phi_2\right|^{p-2}\Phi_2,\quad &&\text{in } \R\times\R^N,\\
				\Delta w&=2\pi\left(|\Phi_1|^2+|\Phi_2|^2\right)\quad &&\text{in } \R^N,
			\end{aligned}
		\end{cases} 
	\end{equation}
	where $\Phi_1,\Phi_2: \R\times\R^{N}\to \mathbb{C}$, $w: \R^{N}\to\R$, $i$ is the imaginary unit, $p>1$ is a constant,  $\mu_{1},\mu_{2},\beta\in \R$ are parameters.

	The system \eqref{evolutional system} arises in the field of quantum mechanics. The first two equations constitute the coupled Schr\"{o}dinger system, typically employed to depict the evolution of the wave functions of two particles. 	Examples also include Bose-Einstein condensation with multiple states \cite{PhysRevLett.78.3594,Frantzeskakis2010,Malomed2008287} and the propagation of mutually incoherent wave packets in nonlinear optics  \cite{Akhmediev19992661,Buryak200263}. The third equation represents the Poisson equation, describing the electrostatic potential. This coupled system \eqref{evolutional system}  is utilized to model the behavior of quantum mechanical systems within a potential field, particularly in semiconductor devices and quantum dots \cite{MRS1990b}, where it captures the interaction between electrons and the electrostatic potential they generate.
	
	Our goal is to find the stationary solution to this system, which takes the form  $\Phi_1(t,x)=e^{i\lambda_{1}t}u(x)$, $\Phi_2(t,x)=e^{i\lambda_{2}t}v(x)$. By substituting these forms into the original system, we obtain the following stationary system:
	\begin{equation}\label{HF}
		\begin{cases}
			\begin{aligned}
				-\Delta u+(\lambda_1+w)u&=\mu_1|u|^{2p-2}u+\beta|u|^{p-2}|v|^{p}u,&&\quad\text{in }\R^N,\\
				-\Delta v+(\lambda_2+w)v&=\mu_2|v|^{2p-2}v+\beta|u|^{p}|v|^{p-2}v,&&\quad\text{in }\R^N,\\
				\Delta w&=2\pi(|u|^2+|v|^2),&&\quad\text{in }\R^N, 
			\end{aligned}
		\end{cases}
	\end{equation}
	where $u, v :\R^{N}\to \R$. 
	
	We are concerned with the case of $N=2$ because the corresponding fundamental solution to the Laplace  equation in this dimension differs significantly from those in higher dimensions, posing substantial challenges in the selection of functional spaces and estimation of energy. In fact,	it is well-known that the fundamental solution to the Laplace  equation is 
	\begin{equation*}
		\Gamma(x)=\begin{cases}
			\frac{1}{N(2-N)\omega_{N}}|x|^{2-N},&N>2,\\
			\frac{1}{2\pi}\log|x|, &N=2,
		\end{cases}
	\end{equation*}
	where $\omega_N$ denotes the volume of unit ball in $\R^N$. Based on this, from the third equation of \eqref{HF} we can represent $w$ as the Newton potential of $2\pi\left(u^2+v^2\right)$, i.e., the convolution with the fundamental solution $\Gamma$. Then \eqref{HF} is equivalent to  
	\begin{equation*}
		\begin{cases}
			-\Delta u+\lambda_1 u+2\pi\left[\Gamma\ast\left(u^2+v^2\right)\right]u=\mu_{1}\left|u\right|^{2p-2}u+\beta\left|v\right|^{p}\left|u\right|^{p-2}u,&\quad \text{in } \R^N,\\
			-\Delta v+\lambda_2 v+2\pi\left[\Gamma\ast\left(u^2+v^2\right)\right]v=\mu_{2}\left|v\right|^{2p-2}v+\beta\left|u\right|^{p}\left|v\right|^{p-2}v,&\quad \text{in } \R^N,
		\end{cases} 
	\end{equation*}
	where $\ast$ denotes the convolution of two functions.

	For the nonlinear Schr\"{o}dinger-Poisson systems of single electron in $\R^3$, there are abundant literature on the topic of the existence of stationary solutions, see \cite{AR2008,AP2008,DM2004,Ruiz2006,DW2005} and the references therein. The energy functionals of these papers are defined on $H^1(\R^3)$. However, it is different for the problem in $\R^2$. In fact, 
	at this time the system is 
	\begin{equation}\label{2-HF}
		\begin{cases}
			-\Delta u+\lambda_1 u+\left[\log|\cdot|\ast\left(u^2+v^2\right)\right]u=\mu_{1}\left|u\right|^{2p-2}u+\beta\left|v\right|^{p}\left|u\right|^{p-2}u,   &\quad \text{in }\R^2,\\
			-\Delta v+\lambda_2 v+\left[\log|\cdot|\ast\left(u^2+v^2\right)\right]v=\mu_{2}\left|v\right|^{2p-2}v+\beta\left|u\right|^{p}\left|v\right|^{p-2}v,   &\quad \text{in } \R^2,
		\end{cases} 
	\end{equation}
	with corresponding energy functional
	\begin{equation*}
		\begin{aligned}
			I_0\left(u,v\right):=&\frac{1}{2}\int_{\R^2}\left(\left|\nabla u\right|^2+\left|\nabla v\right|^2\right)dx+\frac{1}{4}\iint_{\R^2\times\R^2}\log\left|x-y\right|\left(u^2(x)+v^2(x)\right)\left(u^2(y)+v^2(y)\right)dxdy\\
			&+\frac{1}{2}\int_{\R^2}\left(\lambda_1|u|^2+\lambda_2|v|^2\right)dx-\frac{1}{2p}\left(\mu_{1}\int_{\R^2}\left|u\right|^{2p}dx+\mu_{2}\int_{\R^2}\left|v\right|^{2p}dx+2\beta\int_{\R^2}\left|uv\right|^{p}dx\right).
		\end{aligned}
	\end{equation*}
	The appearance of logarithmic convolution makes the energy functional not well-defined on  $H^1(\R^2)$.
	
	To deal with this  obstacle, Stubbe \cite{stubbe2008bound} considered a slightly smaller Hilbert subspace of  $H^1(\R^2)$, namely, 
	\begin{equation}\label{SpaceE}
		E:=\left\{u\in H^1\left(\R^2\right):\int_{\R^2}\log\left(1+|x|^2\right)u^2\left(x\right)dx<\infty\right\}.
	\end{equation}
	This variational framework can be applied to the   planar Schr\"{o}dinger-Poisson system with single electron. Cingolani and Weth \cite{CW16} studied the following equation
	\begin{equation}\label{CWequation}
		-\Delta u + a(x)u+\left[\log|\cdot|\ast u^2\right]u=b|u|^{p-2}u,\quad\text{in }\R^2,
	\end{equation}
	where $b\geq0, p\geq2$ and $a(x)\in L^\infty(\mathbb{R}^2)$ is a positive potential. For the case where $a(x)$ is continuous and $\mathbb{Z}^2$-periodic and $p\geq4$, they proved the existence of a ground state solution and infinitely many geometrically distinct solutions to \eqref{CWequation}.
	If $a(x)$ is a positive constant and the problem \eqref{CWequation} satisfies some symmetric condition, the similar existence and multiplicity results can also be obtained.
	Moreover,  this framework can also be used to study the planar Schr\"odinger-Poisson system with two electrons interaction.  Carvalho et al.   
	\cite{CFMM22} employed the Nehari manifold to establish the existence of a least energy solution for the following system
	\begin{equation}\label{CFMMequation}
		\begin{cases}
			-\Delta u+u+\left[\log|\cdot|\ast\left(u^2+v^2\right)\right]u=\mu_{1}\left|u\right|^{2p-2}u+\beta\left|v\right|^{p}\left|u\right|^{p-2}u,&\quad \text{in } \R^2,\\
			-\Delta v+v+\left[\log|\cdot|\ast\left(u^2+v^2\right)\right]v=\mu_{2}\left|v\right|^{2p-2}v+\beta\left|u\right|^{p}\left|v\right|^{p-2}v,&\quad \text{in } \R^2,
		\end{cases} 
	\end{equation}
	where $p\geq2$ and $\beta\geq0$. They illustrated that the functional related to the problem \eqref{CFMMequation} is well-defined and of class $C^1$ in the product space $E\times E$. Motivated by the previous work, for our problem \eqref{2-HF}, it is reasonable to define the functional $I(u,v)$ on the space $E\times E$  such that it is well-defined and preserves differentiability.   Moreover, we will see that there exist compactly embedding from such working space into Lebesgue's space.  However, since the norm of the space is never invariant by translation, essential difficulties also arise when using the variational method to operate on the space $E\times E$.

	Normalized solutions to differential equations have attracted considerable attention in the literature, particularly because they often arise from physical contexts.  It is known that system \eqref{evolutional system} obeys the mass conservation law, i.e., $\int_{\R^N}|\Phi_1(t,x)|^2dx=c_1$ and $\int_{\R^N}|\Phi_2(t,x)|^2dx=c_2$ for all $t\in\R$, where $c_1$ and  $c_2$ are positive constants. Therefore, the ansatz for $\Phi_1$ and $\Phi_2$ implies that $\int_{\R^2}|u|^2dx=c_1$, $\int_{\R^2}|v|^2dx=c_2$. From physical viewpoint, it is natural to ask whether the system possesses  the solutions with  prescribed  $L^2$-norms.  Here  $\lambda_1$ and $\lambda_2$ can be regarded as Lagrange multipliers which are unknown. For the study of normalized solutions to   Schr\"{o}dinger systems, we refer to \cite{BJS2016,BS2017,BS2019,BZZ2021,GJ2018,JeanjeanZZ2024,mederski2024multiplenormalizedsolutionsnonlinear} and the references therein.  There are also results for the existence of normalized solutions to Schr\"{o}dinger-Poisson systems, such as \cite{BJL2013,JL2013,JLe2021,Lions1987,GH2024}, etc.

	Cingolani and Jeanjean \cite{CJ19} investigated  the existence of normalized solutions for the following planar Schr\"{o}dinger-Poisson system with a single electron
	\begin{equation}\label{CJequation}
		\begin{cases}
			-\Delta u+\lambda u+\gamma\left[\log|\cdot|\ast\left(u^2\right)\right]u=a\left|u\right|^{p-2}u,&\quad \mathrm{in}~{\R^2},\\
			\int_{\R^2}\left|u\right|^2dx=c,
		\end{cases} 
	\end{equation}
	where $c>0,\lambda\in\R,\gamma\in\R,a\in\R$ and $p>2$.
	They obtained the existence of normalized solutions for \eqref{CJequation} in appropriate parameter ranges that ensure the functional is bounded from below, thus leading to a global infimum. Furthermore, they employed the Nehari-Poho\v{z}aev constraint method to locate critical points in other parameter ranges and established both the non-existence and existence conditions for normalized solutions of \eqref{CJequation}. However, to the best of our knowledge, the investigation of normalized solutions of the planar Schr\"{o}dinger-Poisson system with two electrons interaction \eqref{HF} does not appear in the literature. The objective of this paper is to explore the existence of normalized solutions of \eqref{HF} in some appropriate parameter ranges.

	In the present paper, we study the existence of normalized solutions to the system \eqref{HF}. Based on the above discussion, we study the following system:
	\begin{equation}\label{2-NHF}
		\begin{cases}
			-\Delta u+\lambda_1u+\left[\log|\cdot|\ast\left(u^2+v^2\right)\right]u=\mu_{1}\left|u\right|^{2p-2}u+\beta\left|v\right|^{p}\left|u\right|^{p-2}u,&\quad\text{in } \R^2,\\
			-\Delta v+\lambda_2v+\left[\log|\cdot|\ast\left(u^2+v^2\right)\right]v=\mu_{2}\left|v\right|^{2p-2}v+\beta\left|u\right|^{p}\left|v\right|^{p-2}v,&\quad\text{in } \R^2,\\
			\int_{\R^2}\left|u\right|^2dx=c_1,\quad\int_{\R^2}\left|v\right|^2dx=c_2,
		\end{cases} 
	\end{equation}
	where $c_1, c_2>0$. Our goal is to find the  solution $\left(\lambda_1,\lambda_2,u,v\right)\in\R^2\times E\times E$ to the problem \eqref{2-NHF}. Owing to the uniqueness of the solution to Poisson's equation,  $(\lambda_1,\lambda_2,u,v)$ is a solution of \eqref{2-NHF} if and only if $(\lambda_1,\lambda_2,u,v,\log|\cdot|\ast\left(u^2+v^2\right))$ is a solution of \eqref{HF}. We will  constrain the functional $I\left(u,v\right)$ on the product $L^2$ spheres $S\left(c_1\right)\times S\left(c_2\right)$, where
	\begin{equation*}
		\begin{aligned}
			I\left(u,v\right):=&\frac{1}{2}\int_{\R^2}\left(\left|\nabla u\right|^2+\left|\nabla v\right|^2\right)dx+\frac{1}{4}\iint_{\R^2\times\R^2}\log\left|x-y\right|\left(u^2(x)+v^2(x)\right)\left(u^2(y)+v^2(y)\right)dxdy\\
			&-\frac{1}{2p}\left(\mu_{1}\int_{\R^2}\left|u\right|^{2p}dx+\mu_{2}\int_{\R^2}\left|v\right|^{2p}dx+2\beta\int_{\R^2}\left|uv\right|^{p}dx\right),
		\end{aligned}
	\end{equation*}
	and
	\begin{equation*}
		S\left(c\right):=\left\{w\in E:\int_{\R^2}\left|w\right|^2=c>0\right\}.
	\end{equation*}
	That is, the solutions of problem \eqref{2-NHF} can be obtained as critical points of functional $I(u,v)$ subject to the constraint $S\left(c_1\right)\times S\left(c_2\right)$. 
	
	The definition of {\it critical point} of functional $I(u,v)$ is given as follows.
	\begin{definition}\label{definition of critical point}
		We  say that $(u,v)$ is a critical point of functional $I(u,v)$ subjecting to constraint $S\left(c_1\right)\times S\left(c_2\right)$ if  $(u,v)\in S\left(c_1\right)\times S\left(c_2\right)$ and there exists $\left(\lambda_1,\lambda_2\right)\in\R^2$, which appear as Lagrange multipliers, such that 
		\begin{equation*}
			dI|_{S\left(c_1\right)\times S\left(c_2\right)}(u,v)=dI(u,v)+\lambda_1(u,0)+\lambda_2(0,v)=0.
		\end{equation*}
		Here $dI(u,v)$ is the Fr\'{e}chet derivative of $I$ at $(u,v)$.
	\end{definition}
	The set of all critical points of $I$ constrained on $S(c_1)\times S(c_2)$ is denoted by
	\begin{equation*}
		\mathcal{K}(c_1,c_2):=\left\{(u,v):dI|_{S\left(c_1\right)\times S\left(c_2\right)}(u,v)=(0,0)\right\}.
	\end{equation*} 
	Then, we give the definition of {\it ground state solution} and {\it excited state solution} of problem (\ref{2-NHF}).
	\begin{definition}
		We say that $(\lambda_{1},\lambda_{2},u,v)$ is  a ground state solution of problem \eqref{2-NHF} if it satisfies 
		\begin{equation*}	
			dI|_{S\left(c_1\right)\times S\left(c_2\right)}(u,v)=(0,0)\quad\text{\rm{and}}\quad I(u,v)=\inf_{\mathcal{K}(c_1,c_2)}I(u,v),
		\end{equation*}
		or an excited state solution of problem \eqref{2-NHF} if it satisfies 
		\begin{equation*}	
			dI|_{S\left(c_1\right)\times S\left(c_2\right)}(u,v)=(0,0)\quad\text{\rm{and}}\quad I(u,v)>\inf_{\mathcal{K}(c_1,c_2)}I(u,v).
		\end{equation*}
	\end{definition}
	
	For convenience, throughout this paper we always denote that 
	\begin{equation*}
		Q(u)=\int_{\R^2}\left|\nabla u\right|^2dx,\ 
		P(u)=\int_{\R^2}\left|u\right|^{2p}dx=\lVert u\rVert_{2p}^{2p},\ 
		P_0(u,v)=\int_{\R^2}\left|uv\right|^{p}dx,
	\end{equation*}
	\[
	R(u,v)=\mu_{1}P(u)+\mu_{2}P(v)+2\beta P_0(u,v),\] 
	and
	\[
	W_0(u,v)=\iint_{\R^2\times\R^2}\log\left|x-y\right|\left(u^2(x)+v^2(x)\right)\left(u^2(y)+v^2(y)\right)dxdy.
	\]
	Hence, the functional $I(u,v)$ can be represented as
	\begin{align*}
		I\left(u,v\right)&=\frac{1}{2}\left( Q(u)+Q(v)\right) +\frac{1}{4}W_0(u,v)-\frac{1}{2p}[\mu_{1}P(u)+\mu_{2}P(v)+2\beta P_0(u,v)]\\
		&=\frac{1}{2}\left( Q(u)+Q(v)\right) +\frac{1}{4}W_0(u,v)-\frac{1}{2p}R(u,v). 
	\end{align*}
	To show our main results, Gagliardo-Nirenberg inequality is indispensable. It reads as follows:
	\begin{equation}\label{GNieq}
		\lVert u\rVert_{q}^{q}\leqslant K_q\lVert\nabla u\rVert_{2}^{q-2}\lVert u\rVert_{2}^{2},\quad\text{for any }u\in H^1(\R^2).
	\end{equation}
	Here,  $q>2$ is a constant, and $K_{q}>0$ is the best constant of this inequality which depends only on $q$.
	
	The first result of this paper is presented as following:
	\begin{theorem}\label{thm1}
		Let $\beta>0$ and one of the following assumptions hold:
		\begin{enumerate}[\rm(i)]
			\item  $\max\left\{\mu_{1}+\beta,\mu_{2}+\beta\right\}\leqslant 0$;
			\item $1<p<2$;
			\item $p=2$ and $\min\left\{2-K_{4}\left(\mu_{1}+\beta\right)c_1,2-K_{4}\left(\mu_{2}+\beta\right)c_2\right\}>0$.
		\end{enumerate}
		Here $K_{4}$ is the best constant of the Gagliardo-Nirenberg inequality in \eqref{GNieq}. Then  $I$ is bounded from below on $S\left(c_1\right)\times S\left(c_2\right)$, i.e., 
		\[m=\inf_{S\left(c_1\right)\times S\left(c_2\right)}I\left(u,v\right)>-\infty.
		\]
		and $m$ can be attained. Moreover, 	problem \eqref{2-NHF} has  a ground state solution.
	\end{theorem}

	Besides, we also want to know the existence when  parameters fall within other ranges.	For the case when all $\mu_{1},\mu_{2},\beta$ are positive and $p>2$, which is called the mass-supercritical case, we  obtain the following result.
	\begin{theorem}\label{thm2}
		Let $\mu_{1},\mu_{2},\beta>0$  and $p>2$. Assume that
		\begin{equation}\label{energyconstrain}
			c_1+c_2<4^{\frac{p-2}{2p-3}}\left[\frac{p(p-2)^{p-2}}{K_{2p}\mu_{0}(p-1)^{p}}\right]^{\frac{1}{2p-3}},
		\end{equation}
		where $\mu_{0}=\max\left\{\mu_{1}+\beta,\mu_{2}+\beta\right\}$ and $K_{2p}$ is the best constant of the Gagliardo-Nirenberg inequality in \eqref{GNieq}. Then the following hold: 
		\begin{enumerate}[\rm(i)]
			\item There exists a solution  $\left(\lambda_{1}^{+},\lambda_{2}^{+},u^{+},v^{+}\right)$ for problem \eqref{2-NHF} such that it is a local minimizer in
			\begin{equation*}
				\mathcal{A}(c_1,c_2):=\left\{(u,v)\in S(c_1)\times S(c_2):\int_{\R^2}\left(\left|\nabla u\right|+\left|\nabla v\right|^2\right)dx<\frac{p-1}{p-2}\frac{\left(c_1+c_2\right)^2}{4}\right\}.
			\end{equation*}
			Moreover, $\left(\lambda_{1}^{+},\lambda_{2}^{+},u^{+},v^{+}\right)$ is a ground state solution.
			\item There exists an excited state solution  $\left(\lambda_{1}^{-},\lambda_{2}^{-},u^{-},v^{-}\right)$ for problem \eqref{2-NHF}.
		\end{enumerate}
	\end{theorem}

	Although the results presented here for the coupled system are analogous to those in \cite{CJ19} for a single equation, it is important to highlight  that our approach is quite different. Specifically, the challenge arises because our problem involves a coupled system. The compactness discussion in \cite{CJ19} does not apply directly to our case, as the functional is not invariant under different translations on two component functions, respectively. In essence, our approach to overcoming this difficulty is to bound the logarithmic convolution term,  identify a weakly convergent subsequence such that remains invariance of the functional under the same translation on both component functions. We believe that this kind of discussion will be useful for some similar problems involving logarithmic convolution terms.
	
	To deal with Theorem \ref{thm2},  our approach primarily relies on subtle analysis to find a special Palais-Smale sequence at different energy level, which is Inspired by \cite{CJ19,Soave20,DW2024}. Here, we establish a new Poho\v{z}aev identity for the coupled  Schr\"{o}dinger-Poisson system with a logarithmic convolution term, which is crucial to build up the energy estimate. 
	Finally, the classical minimax principle, combining with the Poho\v{z}aev identity and the compactness discussion we establish, permits us to find an excited solution.

	The rest of the paper is organized as follows: section \ref{sec2} introduces the variational setting and some preliminaries. Subsequently, we prove the  compactness and obtain a global minimizer of the functional in section \ref{sec3}.  Section \ref{sec4} is devoted  to describing the geometry of the functional through a certain fiber map and properties of Poho\v{z}aev-Nehari manifold.  Finally, the proof of Theorem \ref{thm2} is established in the section \ref{sec6}. 
	\section{Variational setting and preliminaries}\label{sec2}
	In this section, we will introduce the variational setting and some properties of the functional. Throughout this paper we always assume that $p>1$.
	We first give some notations as follows:
	\begin{itemize}
		\item $\lVert\cdot\rVert_{s}$ denotes the usual norm of $L^{s}(\R^2)$ space.
		\item The Sobolev space $H^1(\R^2)$ is a Hilbert space equipped with the inner product $\langle u,v\rangle_{H^1}=\int_{\R^2}\left(\nabla u\cdot\nabla v+uv\right) dx$ and norm $\lVert u\rVert_{H^1}=\left( \int_{\R^2}\left( \left|\nabla u\right|^2+u^2\right) dx\right) ^{\frac{1}{2}}$.
		\item We denote by $\rightarrow$ and $\rightharpoonup$ the strong convergence and the weak convergence on Hilbert space, respectively.
		\item The relation "$a\lesssim b$" means that there exists a constant $C>0$ such that $a\leqslant Cb$. 
		\item $B_r(x)$ denotes the ball in $\R^2$ with radial $r$ centering on $x$.  
	\end{itemize}

	Let $E$ be the Hilbert space  defined by \eqref{SpaceE}, 
	which is equipped with the inner product $\langle\cdot,\cdot\rangle_E=\langle\cdot,\cdot\rangle_{H^1}+\langle\cdot,\cdot\rangle_0$ and  the norm $\lVert \cdot\rVert_{E}^{2}=\lVert \cdot\rVert_{H^1}^2+\lVert \cdot\rVert_{0}^{2}$. Here, 
	\begin{equation*}
		\langle u,v\rangle_0=\int_{\R^2}\log(1+|x|)u(x)v(x)dx
	\end{equation*}
	and 
	\begin{equation*}
		\lVert u\rVert_0^2=\int_{\R^2}\log(1+|x|)u^2(x)dx
	\end{equation*}
	for any $u,v\in E$. We introduce following symmetric bilinear forms, 
	\begin{equation*}
		\begin{aligned}
			& B_1\left(u,v\right):=\iint_{\R^2\times\R^2}\log\left(1+\left|x-y\right|\right)u(x)v(y)dxdy,\\
			& B_2\left(u,v\right):=\iint_{\R^2\times\R^2}\log\left(1+\frac{1}{\left|x-y\right|}\right)u(x)v(y)dxdy,\\
			& B_0\left(u,v\right):=B_1-B_2=\iint_{\R^2\times\R^2}\log\left|x-y\right|u(x)v(y)dxdy.
		\end{aligned}
	\end{equation*}
	We define the following functionals:
	\begin{equation*}
		\begin{aligned}
			&V_1:H^1\big(\mathbb{R}^2\big)\to[0,\infty],
			&&V_1(u):=B_1\left(u^2,u^2\right)=\iint_{\R^2\times\R^2}\log\left(1+\left|x-y\right|\right)u^2(x)u^2(y)dxdy,\\
			&V_2:L^{\frac83}\big(\mathbb{R}^2\big)\to[0,\infty),
			&&V_2(u):=B_2\left(u^2,u^2\right)=\iint_{\R^2\times\R^2}\log\left(1+\frac{1}{\left|x-y\right|}\right)u^2(x)u^2(y)dxdy,\\	
			&V_0:H^1(\mathbb{R}^2)\to\mathbb{R}\cup\{\infty\},
			&&V_0(u):=V_1(u)-V_2(u)=\iint_{\R^2\times\R^2}\log\left|x-y\right|u^2(x)u^2(y)dxdy.
		\end{aligned}
	\end{equation*}
	The classical Hardy-Littlewood-Sobolev inequality for $\R^2$  (HLS inequality for short)  reads (cf.\cite{LL01b})
	\begin{equation}\label{HLS inequality}
		\left|\iint_{\R^2\times\R^2}f(x)|x-y|^{-\lambda}h(y)dxdy\right|\leqslant C_{\lambda,p}\lVert f\rVert_{p}\lVert h\rVert_{r},\quad\text{for any }f\in L^p(\R^2), h\in L^r(\R^2), 
	\end{equation}
	where $p,r>1$, $0<\lambda<2$ and $\frac{1}{p}+\frac{\lambda}{2}+\frac{1}{r}=2$. Let $p=r=\frac{4}{3}$ and $\lambda=1$ in HLS inequality. Then, combining HLS inequality and the fact that $\log\left(1+s\right)\leqslant s$ for any $s>0$, we can deduce that 
	\begin{align}
		\label{V2}	\left|V_2(u)\right|&\lesssim\lVert u\rVert_{\frac{8}{3}}^{4},\quad \text{for any }u\in L^{\frac{8}{3}}(\R^2),\\
		\label{B2}	\left|B_2\left(u^2,v^2\right)\right|&\lesssim\lVert u\rVert_{\frac{8}{3}}^{2}\lVert v\rVert_{\frac{8}{3}}^{2},\quad \text{for any }u,v\in L^{\frac{8}{3}}(\R^2).
	\end{align}

	Next, we define
	\begin{align*}
		W_1\left(u,v\right):&=\iint_{\R^2\times\R^2}\log\left(1+\left|x-y\right|\right)\left(u^2(x)+v^2(x)\right)\left(u^2(y)+v^2(y)\right)dxdy,\\
		W_2\left(u,v\right):&=\iint_{\R^2\times\R^2}\log\left(1+\frac{1}{\left|x-y\right|}\right)\left(u^2(x)+v^2(x)\right)\left(u^2(y)+v^2(y)\right)dxdy,\\
		W_0\left(u,v\right):&=W_1-W_2=\iint_{\R^2\times\R^2}\log\left|x-y\right|\left(u^2(x)+v^2(x)\right)\left(u^2(y)+v^2(y)\right)dxdy.
	\end{align*}
	A simple calculation leads to 
	\begin{equation}\label{decomposition}
		W_i\left(u,v\right)=V_i(u)+V_i(v)+2B_i\left(u^2,v^2\right),\quad\text{for }i=0,1,2.
	\end{equation}
	Then, from \eqref{V2} and \eqref{B2}, one has
	\begin{equation}\label{W2estimate}
		\left|W_2(u,v)\right|\lesssim\left(\lVert u\rVert_{\frac{8}{3}}^2+\lVert v\rVert_{\frac{8}{3}}^2\right)^2.
	\end{equation}
	By Sobolev embedding , the functional $W_2(u,v)$ is well-defined on $H^{1}(\R^2)\times H^{1}(\R^2)$, but the functional $W_1(u,v)$ is not. 
	Consequently, for our problem, the functional $I(u,v)$ is not well-defined on $H^1(\R^2)\times H^1(\R^2)$. Inspired by \cite{CW16,DW2017,CJ19}, we introduce the product space $E\times E$ equipped with the norm $\left\|(\cdot,\cdot)\right\|^2=\left\|\cdot\right\|_E^2+\left\|\cdot\right\|_E^2$. Indeed, from \eqref{decomposition}, for $u,v\in E$, 
	\begin{equation*}
		\begin{aligned}
			W_1\left(u,v\right)&=V_1(u)+V_1(v)+2B_1\left(u^2,v^2\right)\\
			&=\iint_{\R^2\times\R^2}\log\left(1+\left|x-y\right|\right)\left[ u^2(x)u^2(y)+ v^2(x)v^2(y) + 2u^2(x)v^2(y) \right] dxdy\\
			&\leqslant\iint_{\R^2\times\R^2}\left[\log\left(1+\left|x\right|\right)+\log\left(1+\left|y\right|\right)\right]\left[ u^2(x)u^2(y)+ v^2(x)v^2(y) + 2u^2(x)v^2(y) \right]dxdy\\
			&=2\left( \left\|u\right\|_0^2\left\|u\right\|_2^2+\left\|v\right\|_0^2\left\|v\right\|_2^2\right) +\left\|u\right\|_0^2\left\|v\right\|_2^2+\left\|v\right\|_0^2\left\|u\right\|_2^2.
		\end{aligned}
	\end{equation*}
	Therefore, the functional $W_1(u,v)$ only takes finite value on $E\times E$, and so the energy functional $I$ is well-defined on $E\times E$. Moreover, we show that the logarithmic convolution functional is differentiable on such space, with some additional continuous properties.

	\begin{lemma}\label{lemmaone}
		
		The following properties hold:
		\begin{enumerate}[{\rm (i)}]
			\item The space $E\times E$ is compactly embedded in $L^s(\R^2)\times L^t(\R^2)$ for each $s,t\in\left[2,\infty\right)$;
			\item The functionals $(u,v)\mapsto B_i\left(u^2,v^2\right)$, $i=0,1,2$ and $I$ are of class $C^1(E\times E,\R)$. Moreover, $\langle dB_i\left(u^2,v^2\right),(\phi,\psi)\rangle=2\left( B_i\left(u\phi,v^2\right)+B_i\left(u^2,v\psi\right)\right)$, for $i=0,1,2$;
			\item The functionals $(u,v)\mapsto B_1\left(u^2,v^2\right)$ and $W_1(u,v)$ are weakly lower semicontinuous on $H^1(\R^2)\times H^1(\R^2)$;
			\item The functional $(u,v)\mapsto B_2\left(u^2,v^2\right)$ and $W_2(u,v)$ are continuous on $L^{\frac{8}{3}}(\R^2)\times L^{\frac{8}{3}}(\R^2)$ and weakly continuous on $E\times E$.
		\end{enumerate}
	\end{lemma}
	\begin{proof}
		\hspace*{\fill}
		\begin{enumerate}[{\rm (i)}]
			\item
			From \cite[Lemma 2.2(i)]{CW16}, the space $E$ is compactly embedded into $L^p(\R)$, for each $p\in[2,\infty)$. Then the conclusion is obtained because the product space of two compact spaces is also compact.
			
			\item Computing it directly, we get 
			\begin{equation*}
				\begin{aligned}
					\lim_{t\to0}\frac{B_i\left(\left(u+t\phi\right)^2,\left(v+t\psi\right)^2\right)-B_i\left(u^2,v^2\right)}{t}=2B_i\left(u\phi,v^2\right)+2B_i\left(u^2,v\psi\right).
				\end{aligned}
			\end{equation*}
			Let $\left(u_n,v_n\right)\to(u,v)$ in $E\times E$. Using the H\"older's  inequality and the fact that $\log(1+|x-y|)\leqslant\log(1+|x|)+\log(1+|y|)$, we have that, for every $\phi\in E$,
			\begin{equation*}
				\begin{aligned}
					&\quad\left|B_1(u_n\phi,v_{n}^2)-B_1(u\phi,v^2)\right|\\
					&\leqslant\iint_{\R^2\times\R^2}\log(1+|x-y|)\left|u_n(x)v_{n}^2(y)-u(x)v^{2}(y)\right| \left|\phi(x)\right|dxdy\\
					&\leqslant\iint_{\R^2\times\R^2}\log(1+|x-y|)\left[\left|u_n(x)-u(x)\right|v_{n}^2(y)+\left|v_{n}^2(y)-v^2(y)\right|\left|u(x)\right|\right]\left|\phi(x)\right|dxdy\\
					&\leqslant\iint_{\R^2\times\R^2}\left[\log(1+|x|)+\log(1+|y|)\right]\left[\left|u_n(x)-u(x)\right|v_{n}^2(y)+\left|v_{n}^2(y)-v^2(y)\right|\left|u(x)\right|\right]\left|\phi(x)\right|dxdy\\
					&\leqslant\lVert u_n-u\rVert_{0}\lVert\phi\rVert_{0}\lVert v_n\rVert_{2}^2+\lVert u\rVert_{0}\lVert\phi\rVert_{0}\lVert v_n-v\rVert_{2}\lVert v_n+v\rVert_{2}\\
					&\quad+\lVert u_n-u\rVert_{2}\lVert\phi\rVert_{2}\lVert v_n\rVert_{0}^2+\lVert v_n-v\rVert_{0}\lVert v_n+v\rVert_{0}\lVert u\rVert_{2}\lVert \phi\rVert_{2}\rightarrow0.
				\end{aligned}
			\end{equation*} 
			Similarly, we also obtain that, for each $\psi\in E$,  $\left|B_1\left(u_{n}^2,v_n\psi\right)-B_1\left(u^2,v\psi\right)\right|\to0$ as $n\to\infty$, which indicates that the Gatêaux derivative of $B_1$ is continuous on $E\times E$. Thus, $B_1\left(u^2,v^2\right)$ is of class $C^1(E\times E,\R)$. 
			
			By HLS inequality (\ref{HLS inequality}) and compact embedding(\cite[Lemma 2.2]{CW16}),  one has that, for every $\phi\in E$,
			\begin{align*}
				&\quad\left|B_2(u_n\phi,v_{n}^2)-B_2(u\phi,v^2)\right|\\
				&=\iint_{\R^\times \R^2}\log(1+\frac{1}{|x-y|})\left|u_n(x)v_{n}^2(y)-u(x)v^2(y)\right|\left|\phi(x)\right|dxdy\\
				&\leqslant\iint_{\R^2\times\R^2}\frac{1}{|x-y|}\left[\left|u_n(x)-u(x)\right|v_{n}^2(y)+\left|v_{n}^2(y)-v^2(y)\right|\left|u(x)\right|\right]\left|\phi(x)\right|dxdy\\
				&\lesssim \left(\int_{\R^2}\left|u_n(x)-u(x)\right|^{\frac{4}{3}}\left|\phi(x)\right|^{\frac{4}{3}}\right)^{\frac{3}{4}}\lVert v_n\rVert_{\frac{8}{3}}^2+\left(\int_{\R^2}\left|u(x)\right|^{\frac{4}{3}}\left|\phi(x)\right|^{\frac{4}{3}}dx\right)^{\frac{3}{4}}\left(\int_{\R^2}\left|v_{n}^2(y)-v^2(y)\right|^{\frac{4}{3}}dy\right)^{\frac{3}{4}}\\
				&\lesssim\lVert u_n-u\rVert_{\frac{8}{3}}\lVert\phi\rVert_{\frac{8}{3}}\lVert v_n\rVert_{\frac{8}{3}}^2+\lVert u\rVert_{\frac{8}{3}}\lVert \phi\rVert_{\frac{8}{3}}\lVert v_n+v\rVert_{\frac{8}{3}}\lVert v_n-v\rVert_{\frac{8}{3}}\to 0.
			\end{align*}
			Same as above discussion, $B_2(u^2,v^2)$ is of class $C^1$. Consequently, $B_0(u^2,v^2)$ is of class $C^1(E\times E,\R)$. Then by \cite[Lemma 2.2(ii)]{CW16} and (\ref{decomposition}),  we obtain that $W_i(u,v)$ are of class $C^1(E\times E,\R)$ for $i=0,1,2$, whence $I(u,v)$ is of class $C^1(E\times E,\R)$.
			\item	 Suppose that $(u_n,v_n)\rightharpoonup(u,v)$ in $H^1(\R^2)\times H^1(\R^2)$. Then, up to subsequence, $(u_n,v_n)\to(u,v)$ in $L^2(B_r(0))\times L^2(B_r(0))$ and $(u_n,v_n)\to(u,v)$ a.e. on $B_r(0)\times B_r(0)$ for fixed $r>0$. By Fatou's Lemma, we have 
			\begin{equation*}
				\liminf_{n\to\infty}\iint_{B_r(0)\times B_r(0)}\log(1+|x-y|)u_{n}^{2}(x)v_{n}^{2}(y)dxdy\geqslant\iint_{B_r(0)\times B_r(0)}\log(1+|x-y|)u^2(x)v^2(y)dxdy.
			\end{equation*}
			It follows that 
			\begin{equation*}        \liminf_{n\to\infty}B_1(u_{n}^2,v_{n}^2)\geqslant\iint_{B_r(0)\times B_r(0)}\log(1+|x-y|)u^2(x)v^2(y)dxdy.
			\end{equation*}
			Using monotone convergence theorem, we derive that
			\begin{equation*}
				\lim_{r\to\infty}\iint_{B_r(0)\times B_r(0)}\log(1+|x-y|)u^2(x)v^2(y)dxdy=B_1(u^2,v^2)
			\end{equation*}
			and 
			\begin{equation*}
				\liminf_{n\to\infty}B_1(u_{n}^2,v_{n}^2)\geqslant B_1(u^2,v^2).
			\end{equation*}
			Hence, the functional $(u,v)\to B_1(u_{n}^2,v_{n}^2)$ is weakly lower semicontinuous on $H^1(\R^2)\times H^1(\R^2)$.   From \cite[Lemma2.2(iv)]{CW16} and (\ref{decomposition}), $W_1(u,v)$ is also weakly lower semicontinuous on $H^1(\R^2)\times H^1(\R^2)$.
			
			\item It suffices to use HLS(\ref{B2}) to deduce that the functional $ B_2(u^2,v^2)$ are continuous on $L^{\frac{8}{3}}(\R^2)\times L^{\frac{8}{3}}(\R^2)$. Moreover, combining with (\ref{decomposition}) and \cite[Lemma2.2(iii)]{CW16}, the functional $W_2(u,v)$ is also continuous on $L^{\frac{8}{3}}(\R^2)\times L^{\frac{8}{3}}(\R^2)$. The weak continuity of $B_2(u^2,v^2)$ and $W_2(u,v)$ follows immediately from the compact embedding(Lemma \ref{lemmaone}(i)). \qedhere
			
		\end{enumerate}
	\end{proof}
	
	
		\section{Existence of the global minimizer}\label{sec3}
		\subsection{Compactness results}
		We explore the compactness of approximate sequence of the variational problem. The following lemma is a general version of Lemma 2.1 in \cite{CW16}.
		\begin{lemma}\label{varlemma of CW16}
			Let $\left\{\rho_n\right\}$ be a sequence of positive functions defined on $\R^2$ satisfying that there exist $r>0$ and $\sigma>0$ such that $\int_{B_r(0)}\rho_ndx>\sigma$ for every $n\in \mathbb{N}$. In addition, let $\left\{u_n\right\}$ be a sequence bounded in $L^2(\R^2)$ such that 
			\begin{equation*}
				\sup_{n\in\mathbb{N}}B_1\left(\rho_n,u_{n}^2\right)<+\infty.
			\end{equation*}
			Then $$\int_{\R^2}\log(1+|x|)u_{n}^2(x)dx$$ is bounded for every $n$.
			
			If, moreover,
			\begin{equation*}
				B_1(\rho_n,u_{n}^2)\to 0 \text{ and } \lVert u_n\rVert_2\to 0.
			\end{equation*}
			Then $$\int_{\R^2}\log(1+|x|)u_{n}^2(x)dx\to 0$$ 
			as $n\to \infty$.
		\end{lemma}
		\begin{proof}
			Note that $1+|x-y|\geqslant1+\frac{|y|}{2}\geqslant\sqrt{1+|y|}$		for every $x\in B_r(0)$ and $y\in\R^2\setminus B_{2r}(0)$. Then we have
			\begin{align*}
				B_1\left(\rho_n,u_{n}^2\right)&\geqslant\int_{\R^2\setminus B_{2r}(0)}\int_{B_r(0)}\log(1+|x-y|)\rho_{n}(x)u_{n}^{2}(y)dxdy\\
				&\geqslant\frac{\sigma}{2}\int_{\R^2\setminus B_{2r}(0)}\log(1+|y|)u_{n}^2(y)dy\\
				&\geqslant\frac{\sigma}{2}\left(\lVert u_n\rVert_{0}^2-\int_{B_{2r}(0)}\log(1+|y|)u_{n}^2(y)dy\right) \\
				&\geqslant\frac{\sigma}{2}\left(\lVert u_{n}\rVert_{0}^2-\log(1+2r)\lVert u_n\rVert_{2}^2\right).
			\end{align*}
			Since $B_1\left(\rho_{n},u_{n}^2\right)$ and $\lVert u_n\rVert_2$ are both bounded,  then we obtain that $\lVert u_n\rVert_{0}$ is bounded. If, moreover, $	B_1(\rho_n,u_{n}^2)$ and  $\lVert u_n\rVert_2$ are both tend to zero, $\lVert u_n\rVert_{0}$ also tends to zero.
		\end{proof}
		Next, we illustrate that boundedness of logarithmic convolution implies the concentration of measure. We remark that the following lemma can be extended to any case of $\R^N$, where $N\geqslant2$. Similar result can be found in \cite[Lemma 2.5]{CJ19}.
		\begin{lemma}\label{varlemma of CJ19}
			Let $\left\{\rho_{n}\right\}$, $\left\{\tau_n\right\}$ be two sequences of positive functions such that $\int_{\R^2}\rho_{n}dx=a>0$ and $\int_{\R^2}\tau_ndx=b>0$ for all $n\in\mathbb{N}$. Suppose that $\{B_1(\rho_{n},\tau_n)\}$ is bounded, then  for any $\varepsilon>0$, there exists $\xi(\varepsilon)>0$ such that 
			\begin{equation*}
				\sup_{x\in\R^2}\int_{B_r(x)}\rho_{n}(y)dy\geqslant a-\varepsilon,
			\end{equation*}
			and
			\begin{equation*}
				\sup_{x\in\R^2}\int_{B_r(x)}\tau_{n}(y)dy\geqslant b-\varepsilon,
			\end{equation*}
			for any $r>\xi$ and $n\in\mathbb{N}$.
		\end{lemma}
		\begin{proof}
			Due to the boundedness of $\{B_1(\rho_{n},\tau_n)\}$, we may assume that $0\leqslant B_1\left(\rho_{n},\tau_{n}\right)\leqslant M$ for some positive $M$. Then we can  estimate that
			\begin{equation*}
				\begin{aligned}
					M\geqslant B_1\left(\rho_{n},\tau_{n}\right)&=\iint_{\R^2\times\R^2}\log(1+|x-y|)\rho_{n}(x)\tau_{n}(y)dxdy\\
					&\geqslant\iint_{|x-y|\geqslant r}\log(1+|x-y|)\rho_{n}(x)\tau_{n}(y)dxdy\\
					&\geqslant\log(1+r)\left(ab-\iint_{|x-y|< r}\rho_{n}(x)\tau_{n}(y)dxdy\right)\\
					&=\log(1+r)\left(ab-\int_{\R^2}\left(\int_{B_r(y)}\rho_{n}(x)dx\right)\tau_{n}(y)dy\right)\\
					&\geqslant\log(1+r)b\left(a-\sup_{y\in\R^2}\int_{B_r(y)}\rho_{n}(x)dx\right),
				\end{aligned}
			\end{equation*}
			for every $n\in\mathbb{N}$.  Let $r$ be large enough, then we obtain the conclusion for $\rho_{n}$. Since $B_1(\cdot,\cdot)$ is symmetry bilinear form, we also have the same conclusion for $\tau_{n}$.
		\end{proof}
		
		We show the existence of a weakly convergent sequence.

		\begin{lemma}\label{weak convergent sequence}
			Let $\left\{(u_n,v_n)\right\}\subset S(c_1)\times S(c_2)$ be a bounded sequence in $H^1(\R^2)\times H^1(\R^2)$ such that $\left\{W_1(u_n,v_n)\right\}$ is bounded. Then there exist some $(u,v)\in E\times E$ and a sequence $\left\{x_n\right\}\subset\R^2$  such that  $\left(\tilde{u}_n,\tilde{v}_n\right):=\left(u_{n}(\cdot-x_n),v_{n}(\cdot-x_n)\right) \rightharpoonup(u,v)$ in $S(c_1)\times S(c_2)$, up to a subsequence.
		\end{lemma}
		\begin{proof}
			We take $\rho_{n}=\tau_{n}=u_{n}^2+v_{n}^2$ for any $n\in\mathbb{N}$. Note that
			\begin{equation*}
				B_1(\rho_{n},\tau_{n})=W_1(u_n,v_n)\quad\mathrm{and}\quad \int_{\R^2}\left(u_n^2+v_n^2\right)dx=c_1+c_2.
			\end{equation*}
			Applying Lemma \ref{varlemma of CJ19},  there exist a sequence $\left\{x_n\right\}\subset\R^2$ and some $r>0$ such that 
			\begin{equation*}
				\int_{B_r(x_n)}\left[u_{n}^2(y)+v_{n}^2(y)\right]dy\geqslant\frac{c_1+c_2}{2}>0,\quad\text{ for every }n\in\mathbb{N}.
			\end{equation*}
			Let $(\tilde{u}_n,\tilde{v}_n)=(u_{n}(\cdot-x_n),v_{n}(\cdot-x_n))$. One can easily verify that $W_1(u_n,v_n)=W_1(\tilde{u}_n,\tilde{v}_n)=B_1(\tilde{u}_{n}^2+\tilde{v}_{n}^2,\tilde{u}_{n}^2)+B_1(\tilde{u}_{n}^2+\tilde{v}_{n}^2,\tilde{v}_{n}^2)$. The boundedness of $W_1(u_n,v_n)$ implies that $B_1(\tilde{u}_{n}^2+\tilde{v}_{n}^2,\tilde{u}_{n}^2)$ and $B_1(\tilde{u}_{n}^2+\tilde{v}_{n}^2,\tilde{v}_{n}^2)$ are both bounded. By  Lemma \ref{varlemma of CW16}, we obtain that $\lVert \tilde{u}_{n}\rVert_{0}$ and $\lVert \tilde{v}_{n}\rVert_{0}$ are both bounded. Hence, $\left(\tilde{u}_n,\tilde{v}_n\right)$ is a bounded sequence in $E\times E$. Since $E\times E$ is a Hilbert space, we conclude that there exists $(u,v)\in E\times E$ such that, up to a subsequence, $\left(\tilde{u}_n,\tilde{v}_n\right)\rightharpoonup(u,v)$ in $E\times E$. Finally, due to the compact embedding (Lemma \ref{lemmaone}(i)) and the invariance by translation of the $L^2$-norm, we have $\left(\tilde{u}_n,\tilde{v}_n\right)\rightharpoonup(u,v)\in S(c_1)\times S(c_2)$.
		\end{proof}
		
		The following lemma is useful.
		
		\begin{lemma}\label{lemmafour}
			Assume that $(u_n,v_n)\rightharpoonup (u,v)$ in $E\times E$. Then, up to a subsequence, 
			\begin{equation*}
				\int_{\R^2}\left|u_{n}v_{n}\right|^pdx\rightarrow\int_{\R^2}\left|uv\right|^pdx\quad\text{as }n\rightarrow\infty,\quad\text{for any }p\in(1,\infty).
			\end{equation*}
		\end{lemma}
		\begin{proof}
			Obviously, $u_n\rightharpoonup u$ and $v_n\rightharpoonup v$ in $E$, respectively. Since $E$ is compactly embedded in $L^s(\R^2)$ for $s\in\left[2,\infty\right)$ \cite[Lemma2.2(i)]{CW16}, we have 
			\begin{equation}\label{fomula1}
				\lim_{n\rightarrow\infty}\int_{\R^2}\left|u_n\right|^{2p}dx=\int_{\R^2}\left|u\right|^{2p}dx,\quad\lim_{n\rightarrow\infty}\int_{\R^2}\left|v_n\right|^{2p}dx=\int_{\R^2}\left|v\right|^{2p}dx.
			\end{equation}
			Note that $\{\left|u_n\right|^p\}$ and $\{\left|v_n\right|^p\}$ are both bounded in $L^{2}(\R^2)$. From (\ref{fomula1}), we may assume that, without loss of generality,
			\begin{equation*}
				\left|u_{n}(x)\right|^p\rightarrow\left|u(x)\right|^p\quad \mathrm{and}\quad\left|v_{n}(x)\right|^p\rightarrow\left|v(x)\right|^p\ \text{a.e. in }\R^2.
			\end{equation*}
			Using Brezis-Lieb lemma \cite{BL83}, we have 
			\begin{align*}
				\lim_{n\rightarrow\infty}\int_{\R^2}\left[\left|u_{n}\right|^{2p}-\left|\left|u_{n}\right|^p-\left|u\right|^p\right|^2-\left|u\right|^{2p}\right]dx=0,\\
				\lim_{n\rightarrow\infty}\int_{\R^2}\left[\left|v_{n}\right|^{2p}-\left|\left|v_{n}\right|^p-\left|v\right|^p\right|^2-\left|v\right|^{2p}\right]dx=0.
			\end{align*}
			Therefore, we deduce from (\ref{fomula1}) that
			\begin{align*}
				\lim_{n\rightarrow\infty}\int_{\R^2}\left|\left|u_{n}\right|^p-\left|u\right|^p\right|^2dx=0,\\
				\lim_{n\rightarrow\infty}\int_{\R^2}\left|\left|v_{n}\right|^p-\left|v\right|^p\right|^2dx=0.
			\end{align*}
			Finally, by H\"{o}lder's inequality,  we obtain
			\begin{align*}
				\int_{\R^2}\left[\left|u_{n}v_{n}\right|^p-\left|uv\right|^p\right]dx&=\int_{\R^2}\left[\left|v_{n}\right|^p\left(\left|u_{n}\right|^p-\left|u\right|^p\right)+\left|u\right|^p\left(\left|v_{n_k}\right|^p-\left|v\right|^p\right)\right]dx\\
				&\leqslant\lVert v_{n}\rVert_{2p}^{p}\left(\int_{\R^2}\left(\left|u_{n}\right|^p-\left|u\right|^p\right)^2dx\right)^{\frac{1}{2}}+\lVert u\rVert_{2p}^{p}\left(\int_{\R^2}\left(\left|v_{n}\right|^p-\left|v\right|^p\right)^2dx\right)^{\frac{1}{2}}\rightarrow 0.
			\end{align*}
		\end{proof}

		\subsection{Proof of Theorem \ref{thm1}}
		Now we show the proof of Theorem \ref{thm1}. We will demonstrate that the functional $I(u,v)$ is bounded from below on the constraint $S(c_1)\times S(c_2)$ and its minimizer can be attained.
		\begin{lemma}\label{lemmathree}
			For  $(u,v)\in S(c_1)\times S(c_2)$,  we have the following estimates:
			\begin{enumerate}[{\rm(i)}]
				\item $W_2(u,v)\leqslant C\left(c_{1}^{\frac{3}{4}}Q(u)^{\frac{1}{4}}+c_{2}^{\frac{3}{4}}Q(v)^{\frac{1}{4}}\right)^2$;
				\item $P(u)\leqslant K_{2p}Q(u)^{p-1}c_1$ and $P(v)\leqslant K_{2p}Q(v)^{p-1}c_2$;
				\item $P_0(u,v)\leqslant \frac{K_{2p}}{2}\left(Q(u)^{p-1}c_1+Q(v)^{p-1}c_2\right)$,
			\end{enumerate}
			where $C>0$ is a constant independent of $(u,v)$ and $K_{2p}$ is the best constant of Gagliardo-Nirenberg inequality \eqref{GNieq}.
		\end{lemma}
		\begin{proof}
			Using \eqref{W2estimate} and Gagliardo-Nirenberg inequality (\ref{GNieq}), we obtain (i). (ii) is a direct conclusion of Gagliardo-Nirenberg inequality (\ref{GNieq}). For (iii), by Hölder's inequality, we have $$P_0(u,v)\leqslant\frac{1}{2}\left(P(u)+P(v)\right).$$
			Using Gagliardo-Nirenberg inequality \eqref{GNieq} again, (iii) can be given.
		\end{proof}
		
		\begin{lemma}\label{lemma3.1}
			Under the assumptions of Theorem \ref{thm1}, the functional $I$ is bounded from below on $S\left(c_1\right)\times S\left(c_2\right)$, i.e., $m=\inf_{S\left(c_1\right)\times S\left(c_2\right)}I\left(u,v\right)>-\infty$.
		\end{lemma}
		\begin{proof}
			For the case (i), note first that $\beta>0$. Then, from Hölder's inequality, we get 
			\begin{equation*}
				R(u,v)=\mu_{1}P(u)+\mu_{2}P(v)+2\beta P_0(u,v)\leqslant(\mu_{1}+\beta)P(u)+(\mu_{2}+\beta)P(v).
			\end{equation*}
			Since $\max\left\{\mu_1+\beta,\mu_2+\beta\right\}\leqslant0$, Lemma \ref{lemmathree}(i) leads to
			\begin{equation*}
				I(u,v)\geqslant\frac{1}{2}\left( Q(u)+Q(v)\right)-C\left(c_{1}^{\frac{3}{4}}Q(u)^{\frac{1}{4}}+c_{2}^{\frac{3}{4}}Q(v)^{\frac{1}{4}}\right)^2>-\infty.
			\end{equation*} 
			
			For case (ii), without loss of generality, we may assume that both $\mu_{1}+\beta$ and $\mu_{2}+\beta$ are positive.  Otherwise we can deduce it from the proof of case (i). Since $p<2$, by means of Lemma \ref{lemmathree} we have
			\begin{align*}
				I(u,v)\geqslant&\frac{1}{2}\left( Q(u)+Q(v)\right)-C\left(c_{1}^{\frac{3}{4}}Q(u)^{\frac{1}{4}}+c_{2}^{\frac{3}{4}}Q(v)^{\frac{1}{4}}\right)^2\\
				&-\frac{K_{2p}}{2p}\left[Q(u)^{p-1}c_1\left(\mu_{1}+\beta\right)+Q(v)^{p-1}c_2\left(\mu_{2}+\beta\right)\right]>-\infty.
			\end{align*}

			For case (iii), likewise above procedure, we have
			\begin{equation*}
				I(u,v)\geqslant\left[\frac{1}{2}-\frac{K_4\left(\mu_{1}+\beta\right)}{4}c_1\right]Q(u)+\left[\frac{1}{2}-\frac{K_4\left(\mu_{2}+\beta\right)}{4}c_2\right]Q(v)-C\left(c_{1}^{\frac{3}{4}}Q(u)^{\frac{1}{4}}+c_{2}^{\frac{3}{4}}Q(v)^{\frac{1}{4}}\right)^2>-\infty.
			\end{equation*}
			
		\end{proof}
		\begin{proof}[Proof of Theorem \ref{thm1}]
			In view of Lemma \ref{lemma3.1}, there exists a minimizing sequence $\left\{\left(u_n,v_n\right)\right\}$ of $I$ constrained on $S(c_1)\times S(c_2)$ such that $I(u_n,v_n)\rightarrow m$. Depending on the  estimate of $I(u,v)$ in  Lemma \ref{lemma3.1}, we know that $\{Q(u_n)\}$ and $\{Q(v_n)\}$ are  bounded. Hence, the sequence  $\left\{\left(u_n,v_n\right)\right\}$ is bounded in $H^1(\R^2)\times H^1(\R^2)$. Moreover, it follows from Lemma \ref{lemmathree} that $\{W_2(u_n,v_n)\}, \{P(u_n)\}, \{P(v_n)\}$ and $\{P_0(u_n,v_n)\}$ are bounded. 
			Note that 
			\begin{equation*}
				\frac{1}{4}W_1(u_n,v_n)=I(u_n,v_n)-\frac{1}{2}\left( Q(u_n)+Q(v_n)\right)+\frac{1}{4}W_2(u_n,v_n)+\frac{1}{2p}[\mu_{1}P(u_n)+\mu_{2}P(v_n)+2\beta P_0(u_n,v_n)],
			\end{equation*}
			so  $\{W_1(u_n,v_n)\}$ is bounded. By Lemma \ref{weak convergent sequence},  there exist a sequence  $\left\{x_n\right\}\subset \R^2$ and  $(u,v)\in E\times E$ such that  $\left(\tilde{u}_n,\tilde{v}_n\right)=\left(u_{n}\left(\cdot-x_n\right),v_{n}\left(\cdot-x_n\right)\right)\rightharpoonup\left(u,v\right)$ in $S(c_1)\times S(c_2)$ as $n\to \infty$. Then we have $I(u_n,v_n)=I\left(\tilde{u}_n,\tilde{v}_n\right)$. From Lemma \ref{lemmaone}(iv) and Lemma \ref{lemmafour}, we know that, up to a sunsequence, 
			\begin{equation*}
				W_2(\tilde{u}_n,\tilde{v}_n)\rightarrow W_2(u,v)\quad\text{and}\quad P_0(\tilde{u}_n,\tilde{v}_n)\rightarrow P_0(u,v),\quad\text{as }n\to\infty.
			\end{equation*}
			In addition, we conclude from Lemma \ref{lemmaone}(i)
			that $P(\tilde{u}_n)\rightarrow P(u)$ and $P(\tilde{v}_n)\rightarrow P(v)$ as $n\to\infty$. 
			Due to the weak lower semicontinuity of $Q$ and $W_1$  from Lemma \ref{lemmaone}(iii), we have
			\begin{equation*}
				\begin{aligned}
					m&\leqslant I(u,v)=\frac{1}{2}\left(Q(u)+Q(v)\right)+\frac{1}{4}W_0(u,v)-\frac{1}{2p}R(u,v)\\
					&\leqslant\frac{1}{2}\liminf_{n\to\infty}\left[Q(\tilde{u}_n)+Q(\tilde{v}_n)\right]+\frac{1}{4}\liminf_{n\to\infty}W_1(\tilde{u}_n,\tilde{v}_n)-\frac{1}{4}W_2(\tilde{u}_n,\tilde{v}_n)-\frac{1}{2p}R(\tilde{u}_n,\tilde{v}_n)+o(1)\\            &=\liminf_{n\to\infty}I(\tilde{u}_n,\tilde{v}_n)=\liminf_{n\rightarrow\infty}I\left(u_n,v_n\right)=m,
				\end{aligned}
			\end{equation*}
			which implies that $I(u,v)=m$. 
		\end{proof}
		
		\section{Energy estimates}\label{sec4}
		In following, we assume that $p>2$. 
		Firstly, the Poho\v{z}aev identity for the planar Schr\"{o}dinger-Poisson system is to be given. As far as we know, there is no result devoted to this system. See Appendix A for the proof.

		\begin{lemma}\label{lemma2}
			If $(u,v)\in E\times E$ is a weak solution of the following system:
			\begin{equation}\label{equations of SP}
				\begin{cases}
					-\Delta u+\lambda_1u+\left[\log|\cdot|\ast\left(u^2+v^2\right)\right]u=\mu_{1}\left|u\right|^{2p-2}u+\beta\left|v\right|^{p}\left|u\right|^{p-2}u, \quad &\text{in } \R^2,\\
					-\Delta v+\lambda_2v+\left[\log|\cdot|\ast\left(u^2+v^2\right)\right]v=\mu_{2}\left|v\right|^{2p-2}v+\beta\left|u\right|^{p}\left|v\right|^{p-2}v,  \quad &\text{in } \R^2, 
				\end{cases} 
			\end{equation}
			then $(u,v)$ satisfies the following Pohaev identity:
			\begin{equation*}
				\begin{aligned}
					\lambda_1\int_{\R^2}|u|^2dx+&\lambda_2\int_{\R^2}|v|^2dx+\iint_{\R^2\times\R^2}\log|x-y|\left(u^2(x)+v^2(x)\right)\left(u^2(y)+v^2(y)\right)dxdy\\
					&+\dfrac{1}{4}\left(\int_{\R^2}(|u|^2+|v|^2)dx\right)^2=\dfrac{1}{p}\left(\mu_1\int_{\R^2}|u|^{2p}dx+\mu_2\int_{\R^2}|v|^{2p}dx+2\beta\int_{\R^2}|uv|^pdx\right).
				\end{aligned}
			\end{equation*}
			Moreover, any solution $(u,v)$ of \eqref{equations of SP} in $S(c_1)\times S(c_2)$ satisfies $\M(u,v)=0$, where $\M$ is defined as 
			\begin{equation}\label{Pohožaev-Nehari identity}
				\M\left(u,v\right)=Q(u)+Q(v)-\frac{p-1}{p}\left(\mu_{1}P(u)+\mu_{2}P(v)+2\beta P_0(u,v)\right)-\frac{(c_1+c_2)^2}{4}.
			\end{equation}
		\end{lemma}

		We introduce the fiber map of arbitrary function: $\varphi(t,u)=tu(t\cdot)$ where $t\in \left(0,+\infty\right)$. This map preserves the $L^2$ norm, namely, $\lVert u\rVert_2=\lVert\varphi(t,u)\rVert_2$ for any $t>0$. Moreover, one can also check the following equalities for any $(u,v)\in S(c_1)\times S(c_2)$ and $t>0$:
		\begin{equation}\label{transform}
			\begin{aligned}
				Q\left(\varphi(t,u)\right)=t^2Q(u),&\quad P\left(\varphi(t,u)\right)=t^{2p-2}P(u),\\
				P_0\left(\varphi(t,u),\varphi(t,v)\right)=t^{2p-2}P_0(u,v),&\quad
				R\left(\varphi(t,u),\varphi(t,v)\right)=t^{2p-2}R(u,v),\\
				W_0\left(\varphi(t,u),\varphi(t,v)\right)&=W_0(u,v)-\left(c_1+c_2\right)^2\log t.	
			\end{aligned}
		\end{equation} 
		For any $(u,v)\in S(c_1)\times S(c_2)$, we define the following functions $F_{u,v}, f_{u,v},g_{u,v}:\left(0,+\infty\right)\to\R$ by
		\begin{equation}\label{Function of t}
			\begin{aligned}
				F_{u,v}(t):&=I\left(\varphi(t,u),\varphi(t,v)\right)\\
				&=\frac{t^2}{2}\left(Q(u)+Q(v)\right)+\frac{1}{4}\left[W_0(u,v)-\left(c_1+c_2\right)^2\log t\right]-\frac{t^{2p-2}}{2p}R(u,v).
			\end{aligned}
		\end{equation}
		
		\begin{equation}\label{fuv(t)}
			f_{u,v}(t):=F_{u,v}^{'}(t)=t\left(Q(u)+Q(v)\right)-\frac{\left(c_1+c_2\right)^2}{4t}-\frac{p-1}{p}t^{2p-3}R(u,v),
		\end{equation}
		
		\begin{equation}\label{guv(t)}
			g_{u,v}(t):=t^2f_{u,v}^{'}(t)=t^2\left(Q(u)+Q(v)\right)+\frac{\left(c_1+c_2\right)^2}{4}-\frac{(p-1)(2p-3)}{p}t^{2p-2}R(u,v).
		\end{equation}
		
		Next, we explore some properties of these function. 
		\begin{proposition}\label{proposition}
			Suppose that $A,B,C>0$ and $q>1$. Let  $f(t)=At-\frac{B}{t}-Ct^q$ and $g(t)=At^2+B-Cqt^{q+1}$, where $t\in(0,+\infty)$. Then  there exist $\bar{t}$, $t^{-}$ and $t^{+}$ such that 
			\begin{enumerate}[{\rm(i)}]
				\item $0<t^{+}<\bar{t}<t^{-}$ and  $f(t^{+})=f(t^{-})=0$; 
				\item $f^{'}(t^{+})>0$, $f^{'}(t^{-})<0$;
				\item  $g(t)>0$ on $(0, \bar{t})$ and $g(t)<0$ on $(\bar{t},+\infty)$.
			\end{enumerate}
			if and only if  
			\begin{equation}\label{fomula2}
				\frac{\left(q-1\right)^{\frac{q-1}{2}}}{\left(q+1\right)^{\frac{q+1}{2}}}\cdot\frac{A^{\frac{q+1}{2}}}{B^{\frac{q-1}{2}}}>\frac{C}{2}.
			\end{equation}
		\end{proposition}
		\begin{proof}
			It is easy to check that $f$ has two zeros satisfying (i) and (ii) if and only if  $g\left(\sqrt{\frac{q+1}{q-1}\frac{B}{A}}\right)>0$, which is equivalent to \eqref{fomula2}, and finally note that $g(t)=t^2f^{'}(t)$.
		\end{proof}

		\begin{corollary}\label{corollary2}
			Suppose that $F:(0,+\infty)\to\R$ defined by $F(t):=\frac{1}{2}At^2-B\log t-\frac{C}{q+1}t^{q+1}+W$ where  $A,B,C>0$, $q>1$ satisfying \eqref{fomula2}, and $W$ is a real number. Then $F$ has exactly two critical points in $(0,+\infty)$: the local minimizer $t^{+}$ and the local maximizer $t^{-}$ such that $t^{+}<t^{-}$ and $F(t^{+})<F(t^{-})$.
		\end{corollary}
		\begin{proof}
			Noticing that the derivative of $F(t)$ is $f(t)$ defined in Proposition \ref{proposition}, the conclusions are clear by last proposition.
		\end{proof}

		\begin{lemma}\label{lemma1}
			Let all the assumptions of Theorem \ref{thm2} hold. Then for every $(u,v)\in S(c_1)\times S(c_2)$, $F_{u,v}(t)$ has a unique local minimizer $t_{u,v}^{+}$ and a unique local maximizer $t_{u,v}^{-}$, $t_{u,v}^{+}<t_{u,v}^{-}$, satisfying:
			\begin{enumerate}[\rm(i)]
				\item $F(t_{u,v}^{+})<F(t_{u,v}^{-})$;
				\item $f_{u,v}(t)>0$ on $(t_{u,v}^{+},t_{u,v}^{-})$ and $f_{u,v}(t)<0$ on $(0,t_{u,v}^{+})$ and $(t_{u,v}^{-},+\infty)$;
				\item There exists $\bar{t}_{u,v}$ with $0<t_{u,v}^{+}<\bar{t}_{u,v}<t_{u,v}^{-}$ such that $g_{u,v}(t)>0$ on $(0,\bar{t}_{u,v})$ and $g_{u,v}(t)<0$ on $(\bar{t}_{u,v},+\infty)$.
			\end{enumerate}		 
		\end{lemma}
		\begin{proof}
			Setting $A=Q(u)+Q(v)$, $B=\frac{\left(c_1+c_2\right)^2}{4}$, $C=\frac{p-1}{p}R(u,v)$, $W=\frac{1}{4}W_0(u,v)$ and $q=2p-3$ in Proposition \ref{proposition} and Corollary \ref{corollary2}, it suffices to verify \eqref{fomula2}, i.e.,
			\begin{equation*}
				c_1+c_2<2\left[\frac{p(p-2)^{p-2}}{(p-1)^p}\cdot\frac{\left(Q(u)+Q(v)\right)^{p-1}}{R(u,v)}\right]^{\frac{1}{2p-4}}.
			\end{equation*}
			Due to Lemma \ref{lemmathree} we have that
			\begin{equation}\label{fomulaR}
				\begin{aligned}
					R(u,v)&\leqslant K_{2p}\left[\left(\mu_{1}+\beta\right)Q(u)^{p-1}c_1+\left(\mu_{2}+\beta\right)Q(v)^{p-1}c_2\right]\\
					&\leqslant K_{2p}\mu_{0}\left(Q(u)^{p-1}+Q(v)^{p-1}\right)(c_1+c_2)\\
					&\leqslant K_{2p}\mu_{0}\left(Q(u)+Q(v)\right)^{p-1}(c_1+c_2),
				\end{aligned}
			\end{equation}
			where $\mu_{0}=\max\left\{\mu_{1}+\beta,\mu_{2}+\beta\right\}$. Together with (\ref{energyconstrain}) we obtain that 
			\begin{equation*}
				c_1+c_2<4^{\frac{p-2}{2p-3}}\left[\frac{p(p-2)^{p-2}}{K_{2p}\mu_{0}(p-1)^{p}}\right]^{\frac{1}{2p-3}}\leqslant4^{\frac{p-2}{2p-3}}\left[\frac{p(p-2)^{p-2}}{(p-1)^{p}}\cdot\frac{\left(Q(u)+Q(v)\right)^{p-1}(c_1+c_2)}{R(u,v)}\right]^{\frac{1}{2p-3}},
			\end{equation*} 
			which completes the proof.
		\end{proof}

		Let	\begin{equation}\label{Pohožaev-Nehari set}
			\Omega(c_1,c_2)=\left\{(u,v)\in S(c_1)\times S(c_2):\M\left(u,v\right)=0\right\},
		\end{equation}
		where $\M$ is defined as \eqref{Pohožaev-Nehari identity}.
		In view of \eqref{Function of t}, we find that  $F_{u,v}'(1)=\M(u,v)$. On the other hand, Lemma \ref{lemma2}  illustrates that any critical point of $I(u,v)$ constrained on $S(c_1)\times S(c_2)$ is always contained in $\Omega(c_1,c_2)$. So we divide $\Omega(c_1,c_2)$ into three disjoint parts:
		\begin{equation*}
			\begin{aligned}
				\Omega^{+}(c_1,c_2)=\left\{(u,v)\in S(c_1)\times S(c_2):F_{u,v}^{'}(1)=0,F_{u,v}^{''}(1)>0\right\},\\
				\Omega^{-}(c_1,c_2)=\left\{(u,v)\in S(c_1)\times S(c_2):F_{u,v}^{'}(1)=0,F_{u,v}^{''}(1)<0\right\},\\
				\Omega^{0}(c_1,c_2)=\left\{(u,v)\in S(c_1)\times S(c_2):F_{u,v}^{'}(1)=0,F_{u,v}^{''}(1)=0\right\}.
			\end{aligned}
		\end{equation*}
		By means of \eqref{fuv(t)} and \eqref{guv(t)}, the set $\Omega^{+}$, $\Omega^{-}$ and $\Omega^{0}$ can be represented equivalently as 
		\begin{equation}\label{Omega}
			\begin{aligned}
				\Omega^{+}(c_1,c_2)=\left\{(u,v)\in S(c_1)\times S(c_2):f_{u,v}(1)=0,g_{u,v}(1)>0\right\},\\
				\Omega^{-}(c_1,c_2)=\left\{(u,v)\in S(c_1)\times S(c_2):f_{u,v}(1)=0,g_{u,v}(1)<0\right\},\\
				\Omega^{0}(c_1,c_2)=\left\{(u,v)\in S(c_1)\times S(c_2):f_{u,v}(1)=0,g_{u,v}(1)=0\right\}.
			\end{aligned}
		\end{equation}
		
		\begin{lemma}\label{lemma3}
			Let $(u,v)\in S(c_1)\times S(c_2)$, then
			\begin{equation}\label{Omega+-}		  \left(\varphi(t_{u,v}^{+},u),\varphi(t_{u,v}^{+},v)\right)\in\Omega^{+}(c_1,c_2)\quad\text{and }\quad	\left(\varphi(t_{u,v}^{-},u),\varphi(t_{u,v}^{-},v)\right)\in\Omega^{-}(c_1,c_2),
			\end{equation} 
			where $t_{u,v}^{+}$ and $t_{u,v}^{-}$ are obtained in Lemma {\rm\ref{lemma1}}. Moreover, $\Omega^{0}(c_1,c_2)=\varnothing$.
		\end{lemma}
		\begin{proof}
			Suppose $(u,v)\in S(c_1)\times S(c_2)$. From \eqref{transform}, for every $t\in(0,+\infty)$ we have
			\begin{equation*}
				f_{\varphi(t,u),\varphi(t,v)}(1)=tf_{u,v}(t),\quad g_{\varphi(t,u),\varphi(t,v)}(1)=g_{u,v}(t).
			\end{equation*}
			Then, by Lemma \ref{lemma1} and \eqref{Omega}, we know \eqref{Omega+-} holds. 
			If there exists $(u',v')\in\Omega^{0}(c_1,c_2)$, then by $f_{u',v'}(1)=0$ and Lemma \ref{lemma1} we have that either $t_{u',v'}^{+}=1$ or $t_{u',v'}^{-}=1$. However, both $g_{u',v'}(t_{u',v'}^{+})$ and $g_{u',v'}(t_{u',v'}^{-})$ are not vanishing by Lemma \ref{lemma1}(iii), which leads to a contradiction.
		\end{proof}
		\begin{remark}\label{unique interation}
			\rm{From Lemma \ref{lemma1}, for every $(u,v)\in S(c_1)\times S(c_2)$, the map $(u,v)\mapsto(\varphi(t,u),\varphi(t,v))$ intersects $\Omega^{+}(c_1,c_2)$ uniquely at $(\varphi(t_{u,v}^{+},u),\varphi(t_{u,v}^{+},v))$ and $\Omega^{-}(c_1,c_2)$ uniquely at $(\varphi(t_{u,v}^{-},u),\varphi(t_{u,v}^{-},v))$.}
		\end{remark}
		\begin{lemma}\label{lemma4}
			Under the assumption of Theorem \ref{thm2}, we have that 
			\begin{equation*}
				Q(u)+Q(v)<\frac{p-1}{p-2}\frac{\left(c_1+c_2\right)^2}{4},\quad\text{if }(u,v)\in\Omega^{+}(c_1,c_2)
			\end{equation*}
			and
			\begin{equation*}
				Q(u)+Q(v)>\frac{p-1}{p-2}\frac{\left(c_1+c_2\right)^2}{4},\quad\text{if }(u,v)\in\Omega^{-}(c_1,c_2).
			\end{equation*}
		\end{lemma}
		\begin{proof}
			According to the definition of $\Omega^{+}(c_1,c_2)$ in (\ref{Omega}), inequality $g_{u,v}(1)>0$ and equality $f_{u,v}(1)=0$ imply the first part. The second part is similar.
		\end{proof}
		Set 
		\begin{equation*}
			\mathcal{A}(c_1,c_2)=\left\{(u,v)\in S(c_1)\times S(c_2):Q(u)+Q(v)<\frac{p-1}{p-2}\frac{\left(c_1+c_2\right)^2}{4}\right\},
		\end{equation*}
		and introduce the following notations:
		\begin{equation}\label{minimumeneergy}
			m_1=\inf_{\Omega^{+}(c_1,c_2)}I(u,v),\quad
			\tilde{m}=\inf_{\mathcal{A}(c_1,c_2)}I(u,v),\quad m_2=\inf_{\Omega^{-}(c_1,c_2)}I(u,v).
		\end{equation}
		Then we have following relation among the above quantities.
		\begin{lemma}\label{lemma5}
			Under the assumption of Theorem \ref{thm2}, we have $m_1=\tilde{m}\leqslant m_2$. Moreover, if $m_2$ can be attained, then the inequality is strict.
		\end{lemma}
		\begin{proof}
			For any $\varepsilon>0$, let $(u,v)\in \Omega^{-}(c_1,c_2)$ satisfying $I(u,v)<m_2+\varepsilon$. From Lemma \ref{lemma1}(i) and Lemma \ref{lemma3} we know that $t_{u,v}^{+}<t_{u,v}^{-}=1$ and
			\begin{equation*}
				m_1\leqslant F_{u,v}(t_{u,v}^{+})<F_{u,v}(t_{u,v}^{-})=F_{u,v}(1)=I(u,v)<m_2+\varepsilon.
			\end{equation*}
			We have $m_1\leqslant m_2$ due to the arbitrariness of $\varepsilon$. If there exists a $(u,v)\in\Omega^{-}(c_1,c_2)$ such that $I(u,v)=m_2$, then $m_1<I(u,v)=m_2$.
			
			Next we prove $m_1=\tilde{m}$. According to Lemma \ref{lemma4}, it follows that $\Omega^{+}(c_1,c_2)\subset\mathcal{A}(c_1,c_2)$, which implies $\tilde{m}\leqslant m_1$. Fix $(u,v)\in \mathcal{A}(c_1,c_2)$. By increasing of $Q(\varphi(t,u))+Q(\varphi(t,v))$ with respect to $t$ and the definition of $\mathcal{A}(c_1,c_2)$, there exists $t_0>0$ such that  $\left(\varphi(t,u),\varphi(t,v)\right)\in\mathcal{A}(c_1,c_2)$ for every $t\in(0,t_0)$, while $\left(\varphi(t,u),\varphi(t,v)\right)\notin\mathcal{A}(c_1,c_2)$ for every $t\in(t_0,+\infty)$. Obviously, $t_0>1$. In virtue of Lemma \ref{lemma3} and Lemma \ref{lemma4}, one can check that $t_{u,v}^{+}<t_0<t_{u,v}^{-}$. By Lemma \ref{lemma1}(ii), we obtain that $F_{u,v}(t_{u,v}^{+})$ is the unique local minimum of $F_{u,v}(t)$ in $(0,t_0)$.  Hence, from Lemma \ref{lemma3}, it holds $I(u,v)=F_{u,v}(1)\geqslant F_{u,v}(t_{u,v}^{+})=I\left(\varphi(t_{u,v}^{+},u),\varphi(t_{u,v}^{+},v)\right)\geqslant m_1$. Therefore, we have $\tilde{m}\geqslant m_1$ due to the arbitrariness of $(u,v)$.
		\end{proof}
		
		\section{Existence result for the mass-supercritical case }\label{sec6}
		In this section, we prove the existence of a solution of local minimum type in $\Omega^{+}(c_1,c_2)$ and a  solution of minimax type in $\Omega^{-}(c_1,c_2)$, respectively.
		\subsection{Existence of a ground state solution}
		We establish a lower bound of $I$ on $\mathcal{A}(c_1,c_2)$. Recall that $\tilde{m}=\inf_{\mathcal{A}(c_1,c_2)}I(u,v)$ in \eqref{minimumeneergy}.
		\begin{lemma}\label{finite of m}
			$\tilde{m}>-\infty$.
		\end{lemma}
		\begin{proof}
			From \eqref{fomulaR} it holds that 
			\begin{equation*}
				\begin{aligned}
					I(u,v)&\geqslant\frac{1}{2}\left(Q(u)+Q(v)\right)-\frac{1}{4}W_2(u,v)-\frac{1}{2p}R(u,v)\\
					&\geqslant\frac{1}{2}\left(Q(u)+Q(v)\right)-C\left(c_{1}^{\frac{3}{4}}Q(u)^{\frac{1}{4}}+c_{2}^{\frac{3}{4}}Q(v)^{\frac{1}{4}}\right)^{2}-K_{2p}\mu_0\left(Q(u)+Q(v)\right)^{p-1}(c_1+c_2).
				\end{aligned}
			\end{equation*}
			According to the definition of $\mathcal{A}(c_1,c_2)$, it follows from above inequality that $\tilde{m}>-\infty$.
		\end{proof}
		
		\begin{lemma}\label{minimizermathcalA}
			There exists a  $(u^{+},v^{+})\in\mathcal{A}(c_1,c_2)$ such that $I(u^{+},v^{+})=\tilde{m}$.
		\end{lemma}
		\begin{proof}
			Let $\left\{(u_n,v_n)\right\}\subset\mathcal{A}(c_1,c_2)$ be a minimizing sequence  of $I$. Since $\{Q(u_n)+Q(v_n)\}$ is bounded, a similar argument as in the proof of Theorem \ref{thm1} implies conclusion. 
		\end{proof}
		
		\begin{proof}[Proof of Theorem \ref{thm2}(i)]
			In Lemma \ref{minimizermathcalA}, we obtain a local minimizer $(u^{+},v^{+})\in\mathcal{A}(c_1,c_2)$ and it is truly a critical point of $I$ constrained on $S(c_1)\times S(c_2)$. It remains to show that it is a ground state solution. According to Lemma \ref{lemma2} and Lemma \ref{lemma4}, we get that $(u^{+},v^{+})\in\Omega^{+}(c_1,c_2)$ and $\mathcal{K}(c_1,c_2)\subset\Omega(c_1,c_2)$. Then we have  $\inf_{\Omega(c_1,c_2)}I\leqslant\inf_{\mathcal{K}(c_1,c_2)}I$. Moreover, it follows from Lemma \ref{lemma3} and Lemma \ref{lemma5} that $\tilde{m}=m_1=\inf_{\Omega(c_1,c_2)}I$. Then Lemma \ref{minimizermathcalA} leads to the following:
			\begin{equation*}
				\inf_{\mathcal{K}(c_1,c_2)}I\leqslant I(u^{+},v^{+})=\tilde{m}=m_1=\inf_{\Omega(c_1,c_2)}I\leqslant\inf_{\mathcal{K}(c_1,c_2)}I,
			\end{equation*}
			which implies that $I(u^{+},v^{+})=\inf_{\mathcal{K}(c_1,c_2)}I$.
		\end{proof}

		\subsection{Existence of an excited state solution}
		We introduce the definition of homotopy-stable family, which can be found in \cite[Definition 3.1]{Ghou93b}.
		\begin{definition}\label{GhoussobDefinition}
			Let $B$ be a closed subset of topological space $X$. We shall say that a class $\mathcal{F}$ of compact subsets of $X$ is a homotopy-stable family with boundary $B$ if:
			\begin{enumerate}[{\rm(i)}]
				\item every set in $\mathcal{F}$ contains $B$;
				\item for any set $A$ in $\mathcal{F}$ and any $\eta\in C([0,1]\times X;X)$ satisfying $\eta(t,x)=x$ for all $(t,x)$ in $\left(\left\{0\right\}\times X\right)\cup\left([0,1]\times B\right)$ we have that $\eta(\left\{1\right\}\times A)\in\mathcal{F}$.
			\end{enumerate}
		\end{definition}
		We note that $B=\varnothing$ is admissible in above definition. 
		\begin{lemma}[\cite{Ghou93b}, Theorem 3.2]\label{Minimaxthm}
			Let $J$ be a $C^1$-functional on a complete connected $C^1$-Finsler manifold $X$(without boundary) and consider a homotopy-stable family $\mathcal{F}$ with a closed boundary $B$. Assume
			\begin{equation*}
				c=c(J,\mathcal{F})=\inf_{A\in\mathcal{F}}\max_{x\in A}J(x)
			\end{equation*}
			is finite and suppose that 
			\begin{equation}\label{minimaxF1}
				\sup_{B}J<c.
			\end{equation}
			Then, for any sequence of sets $\left\{A_n\right\}$ in $\mathcal{F}$ such that $\lim_{n\to\infty}\sup_{A_n}J=c$, there exists a sequence $\left\{x_n\right\}$ in $X$ such that 
			\begin{enumerate}[{\rm(i)}]
				\item $\lim_{n\to\infty}J(x_n)=c$;
				\item $\lim_{n\to\infty}dJ(x_n)=0$;
				\item $\lim_{n\to\infty}\text{\rm{dist}}(x_n,A_n)=0$.
			\end{enumerate}
		\end{lemma}
		
		Let us introduce the following auxiliary functional:
		\begin{equation*}
			I^{-}(u,v):=I\left(\varphi(t_{u,v}^{-},u),\varphi(t_{u,v}^{-},v)\right)
		\end{equation*}
		for every $(u,v)\in S(c_1)\times S(c_2)$. We denote $T_{(u,v)}(S(c_1)\times S(c_2))$ as the tangent space of $S(c_1)\times S(c_2)$ at $(u,v)$.
		\begin{lemma}\label{lemme1}
			The map 
			$$S(c_1)\times S(c_2)\to\R,\quad(u,v)\mapsto t_{u,v}^{-}$$
			is of $C^1$ class.
		\end{lemma}
		\begin{proof}
			Define functional $\Psi:(0,+\infty)\times S(c_1)\times S(c_2)\to\R$ by $\Psi(t,u,v):=f_{u,v}(t)$, where $f_{u,v}$ comes from \eqref{fuv(t)}. It is easy to check that $\Psi$ is of $C^1$ class. By Lemma \ref{lemma1}, we have  $\Psi(t_{u,v}^{-},u,v)=0$ and $\partial_{t}\Psi(t_{u,v}^{-},u,v)=g_{u,v}(t_{u,v}^{-})/(t_{u,v}^{-})^2<0$, where $g_{u,v}$ is defined as \eqref{guv(t)}. By implicit function theorem we obtain the conclusion.
		\end{proof}

		\begin{lemma}\label{lemme2}
			For any $(u,v)\in S(c_1)\times S(c_2)$, the map $$T_{(u,v)}(S(c_1)\times S(c_2))\to T_{(\varphi(t_{u,v}^{-},u),\varphi(t_{u,v}^{-},v))}(S(c_1)\times S(c_2)),\quad(\phi,\psi)\mapsto(\varphi(t_{u,v}^{-},\phi),\varphi(t_{u,v}^{-},\psi))$$
			is isomorphism.
		\end{lemma}
		\begin{proof}
			For any $(\phi,\psi)\in T_{(u,v)}(S(c_1)\times S(c_2))$ and $t\in(0,+\infty)$, we have 
			\begin{equation*}
				\int_{\R^2}\varphi(t,u)(x)\varphi(t,\phi)(x)dx=\int_{\R^2}t^2u(tx)\phi(tx)dx=\int_{\R^2}u(x)\phi(x)dx=0,
			\end{equation*}
			and
			\begin{equation*}
				\int_{\R^2}\varphi(t,v)(x)\varphi(t,\psi)(x)dx=\int_{\R^2}t^2v(tx)\psi(tx)dx=\int_{\R^2}v(x)\psi(x)dx=0.
			\end{equation*}
			Then $(\varphi(t_{u,v}^{-},\phi),\varphi(t_{u,v}^{-},\psi))\in T_{(\varphi(t_{u,v}^{-},u),\varphi(t_{u,v}^{-},v))}(S(c_1)\times S(c_2))$ and the map is well-defined. Clearly it is linear and the remained proof is standard (see \cite[Lemma 3.6]{BS2019}). 
		\end{proof}
		
		\begin{lemma}\label{lemme3}
			$I^{-}$ is of $C^1$ class and 
			\begin{equation*}
				\langle dI^{-}|_{S(c_1)\times S(c_2)}(u,v),(\phi,\psi)\rangle=\langle dI|_{S(c_1)\times S(c_2)}(\varphi(t_{u,v}^{-},u),\varphi(t_{u,v}^{-},v)),(\varphi(t_{u,v}^{-},\phi),\varphi(t_{u,v}^{-},\psi))\rangle
			\end{equation*}
			for any $(u,v)\in S(c_1)\times S(c_2)$ and $(\phi,\psi)\in T_{(u,v)}(S(c_1)\times S(c_2))$.
		\end{lemma}
		\begin{proof}
			This proof is inspired by \cite[Lemma 3.15]{CJ19}. Let $(u,v)\in S(c_1)\times S(c_2)$ and $(\phi,\psi)\in T_{(u,v)}(S(c_1)\times S(c_2))$. Then $(\phi,\psi)=\gamma'(0)$, where $\gamma:(-\eta,\eta)\to S(c_1)\times S(c_2)$, $\gamma(\varepsilon):=(u_{\varepsilon},v_{\varepsilon})$ is a $C^1$ curve satisfying $\gamma(0)=(u,v)$. For convenience, we denote that $\gamma'(\varepsilon)=(u_{\varepsilon}',v_{\varepsilon}')$, so that $(u_0',v_0')=(\phi,\psi)$. Set that $t_{\varepsilon}^{-}=t_{u_{\varepsilon},v_{\varepsilon}}^{-}$ and  $t_{0}^{-}=t_{u,v}^{-}$. Consider the incremental quotient
			\begin{equation}\label{formula2}
				\frac{I^{-}(\gamma(\varepsilon))-I^{-}(\gamma(0))}{\varepsilon}=\frac{F_{u_{\varepsilon},v_{\varepsilon}}(t_{\varepsilon}^{-})-F_{u,v}(t_{0}^{-})}{\varepsilon}.
			\end{equation}
			Note that $t_{\varepsilon}^{-}$ is a local maximizer of $F_{u_{\varepsilon},v_{\varepsilon}}(t)$ (see Lemma \ref{lemma1}). Then, for sufficiently small $\varepsilon$, we obtain through mean value theorem that 
			\begin{align*}
				&F_{u_{\varepsilon},v_{\varepsilon}}(t_{\varepsilon}^{-})-F_{u,v}(t_{0}^{-})\geqslant F_{u_{\varepsilon},v_{\varepsilon}}(t_{0}^{-})-F_{u,v}(t_{0}^{-})\\
				=&\frac{(t_{0}^{-})^2}{2}\left(Q(u_{\varepsilon})+Q(v_{\varepsilon})-Q(u)-Q(v)\right)+\frac{1}{4}\left(W_0(u_{\varepsilon},v_{\varepsilon})-W_0(u,v)\right)-\frac{(t_{0}^{-})^{2p-2}}{2p}\left(R(u_{\varepsilon},v_{\varepsilon})-R(u,v)\right)\\
				=&(t_{0}^{-})^2\varepsilon\int_{\R^2}\left(\nabla u_{\tau_{1}\varepsilon}\cdot\nabla u_{\tau_{1}\varepsilon}'+\nabla v_{\tau_{1}\varepsilon}\cdot\nabla v_{\tau_{1}\varepsilon}'\right)dx+\varepsilon B_{0}\left(u_{\tau_{2}\varepsilon}^2+v_{\tau_{2}\varepsilon}^2,u_{\tau_{2}\varepsilon}u_{\tau_{2}\varepsilon}'+v_{\tau_{2}\varepsilon}v_{\tau_{2}\varepsilon}'\right)\\
				&-\varepsilon(t_{0}^{-})^{2p-2}\int_{\R^2}\left(\mu_1|u_{\tau_{3}\varepsilon}|^{2p-2}u_{\tau_{3}\varepsilon}u_{\tau_{3}\varepsilon}'+\mu_2|v_{\tau_{3}\varepsilon}|^{2p-2}v_{\tau_{3}\varepsilon}v_{\tau_{3}\varepsilon}'\right)dx\\
				&-2\varepsilon(t_{0}^{-})^{2p-2}\beta\int_{\R^2}\left(|u_{\tau_{3}\varepsilon}|^{p-2}|v_{\tau_{3}\varepsilon}|^{p}u_{\tau_{3}\varepsilon}u_{\tau_{3}\varepsilon}'+|u_{\tau_{3}\varepsilon}|^{p}|v_{\tau_{3}\varepsilon}|^{p-2}v_{\tau_{3}\varepsilon}v_{\tau_{3}\varepsilon}'\right)dx,
			\end{align*}
			where $\tau_1,\tau_2,\tau_3\in(0,1)$.
			Similarly, since $t_{0}^{-}$ is a local maximizer of $F_{u,v}(t)$, it holds
			\begin{align*}
				&F_{u_{\varepsilon},v_{\varepsilon}}(t_{\varepsilon}^{-})-F_{u,v}(t_{0}^{-})\leqslant F_{u_{\varepsilon},v_{\varepsilon}}(t_{\varepsilon}^{-})-F_{u,v}(t_{\varepsilon}^{-})\\
				=&\frac{(t_{\varepsilon}^{-})^2}{2}\left(Q(u_{\varepsilon})+Q(v_{\varepsilon})-Q(u)-Q(v)\right)+\frac{1}{4}\left(W_0(u_{\varepsilon},v_{\varepsilon})-W_0(u,v)\right)-\frac{(t_{\varepsilon}^{-})^{2p-2}}{2p}\left(R(u_{\varepsilon},v_{\varepsilon})-R(u,v)\right)\\
				=&(t_{\varepsilon}^{-})^2\varepsilon\int_{\R^2}\left(\nabla u_{\tau_{4}\varepsilon}\cdot\nabla u_{\tau_{4}\varepsilon}'+\nabla v_{\tau_{4}\varepsilon}\cdot\nabla v_{\tau_{4}\varepsilon}'\right)dx+\varepsilon B_{0}\left(u_{\tau_{5}\varepsilon}^2+v_{\tau_{5}\varepsilon}^2,u_{\tau_{5}\varepsilon}u_{\tau_{5}\varepsilon}'+v_{\tau_{5}\varepsilon}v_{\tau_{5}\varepsilon}'\right)\\
				&-\varepsilon(t_{\varepsilon}^{-})^{2p-2}\int_{\R^2}\left(\mu_1|u_{\tau_{6}\varepsilon}|^{2p-2}u_{\tau_{6}\varepsilon}u_{\tau_{6}\varepsilon}'+\mu_2|v_{\tau_{6}\varepsilon}|^{2p-2}v_{\tau_{6}\varepsilon}v_{\tau_{6}\varepsilon}'\right)dx\\
				&-2\varepsilon(t_{\varepsilon}^{-})^{2p-2}\beta\int_{\R^2}\left(|u_{\tau_{6}\varepsilon}|^{p-2}|v_{\tau_{6}\varepsilon}|^{p}u_{\tau_{6}\varepsilon}u_{\tau_{6}\varepsilon}'+|u_{\tau_{6}\varepsilon}|^{p}|v_{\tau_{6}\varepsilon}|^{p-2}v_{\tau_{6}\varepsilon}v_{\tau_{6}\varepsilon}'\right)dx,
			\end{align*}
			where $\tau_4,\tau_5,\tau_6\in(0,1)$. Due to Lemma \ref{lemme1} and \eqref{formula2}, one has that 
			\begin{align*}
				&\lim_{\varepsilon\to 0}\frac{I^{-}(\gamma(\varepsilon))-I^{-}(\gamma(0))}{\varepsilon}\\
				&=(t_{0}^{-})^2\int_{\R^2}\left(\nabla u\cdot\nabla \phi+\nabla v\cdot\nabla\psi\right)dx+B_{0}\left(u^2+v^2,u\phi+v\psi\right)\\
				&\quad+(t_{0}^{-})^{2p-2}\int_{\R^2}\left[\mu_1|u|^{2p-2}u\phi+\mu_2|v|^{2p-2}v\psi+2\beta\left(|u|^{p-2}|v|^{p}u\phi+|u|^{p}|v|^{p-2}v\psi\right)\right]dx\\
				&=\int_{\R^2}\left(\nabla\varphi(t_{0}^{-},u)\cdot\nabla\varphi(t_{0}^{-},\phi)+\nabla\varphi(t_{0}^{-},v)\cdot\nabla\varphi(t_{0}^{-},\psi)\right)dx\\
				&\quad+B_{0}\left(\varphi^2(t_{0}^{-},u)+\varphi^2(t_{0}^{-},v),\varphi(t_{0}^{-},u)\varphi(t_{0}^{-},\phi)+\varphi(t_{0}^{-},v)\varphi(t_{0}^{-},\psi)\right)\\
				&\quad+\log t_{0}^{-}\int_{\R^2}\left(u^2(x)+v^2(x)\right)dx\int_{\R^2}\left(u(y)\phi(y)+v(y)\psi(y)\right)dy\\
				&\quad-\int_{\R^2}\left(\mu_1|\varphi(t_{0}^{-},u)|^{2p-2}\varphi(t_{0}^{-},u)\varphi(t_{0}^{-},\phi)+\mu_2|\varphi(t_{0}^{-},v)|^{2p-2}\varphi(t_{0}^{-},v)\varphi(t_{0}^{-},\psi)\right)dx\\
				&\quad-2\beta\int_{\R^2}\left(|\varphi(t_{0}^{-},u)|^{p-2}|\varphi(t_{0}^{-},v)|^{p}\varphi(t_{0}^{-},u)\varphi(t_{0}^{-},\phi)+|\varphi(t_{0}^{-},u)|^{p}|\varphi(t_{0}^{-},v)|^{p-2}\varphi(t_{0}^{-},v)\varphi(t_{0}^{-},\psi)\right)dx\\
				&=\langle dI|_{S(c_1)\times S(c_2)}(\varphi(t_{0}^{-},u),\varphi(t_{0}^{-},v)),(\varphi(t_{0}^{-},\phi),\varphi(t_{0}^{-},\psi))\rangle\\
				&\quad+\log t_{0}^{-}\int_{\R^2}\left(u^2(x)+v^2(x)\right)dx\int_{\R^2}\left(u(y)\phi(y)+v(y)\psi(y)\right)dy\\
				&=\langle dI|_{S(c_1)\times S(c_2)}(\varphi(t_{u,v}^{-},u),\varphi(t_{u,v}^{-},v)),(\varphi(t_{u,v}^{-},\phi),\varphi(t_{u,v}^{-},\psi))\rangle.
			\end{align*}
		\end{proof}

		\begin{lemma}\label{lemme4}
			Let $\mathcal{F}$ be a homotopy-stable family of compact subsets of $S(c_1)\times S(c_2)$ with closed boundary $B$ and
			\begin{equation*}
				e_{\mathcal{F}}^{-}:=\inf_{A\in\mathcal{F}}\sup_{(u,v)\in A}I^{-}(u,v).
			\end{equation*}
			Suppose that $B$ is contained in a connected component of $\Omega^{-}(c_1,c_2)$, and $\sup_{B}I^{-}<e_{\mathcal{F}}^{-}<+\infty$. Then there exists a Palais-Smale sequence $\left\{(u_n,v_n)\right\}\subset\Omega^{-}(c_1,c_2)$ for $I$ restricted to $S(c_1)\times S(c_2)$ at level $e_{\mathcal{F}}^{-}$. 
		\end{lemma}
		\begin{proof}
			We follow ideas from \cite[Lemma 3.16]{CJ19}. Let $\left\{D_{n}\right\}\subset \mathcal{F}$ such that $\max_{(u,v)\in D_{n}}I^{-}<e_{\mathcal{F}}^{-}+\frac{1}{n}$. We define the following deformation:
			\begin{equation*}
				\alpha:[0,1]\times S(c_1)\times S(c_2)\to S(c_1)\times S(c_2),\quad\alpha(s,u,v)=(\varphi(1-s+st_{u,v}^{-},u),\varphi(1-s+st_{u,v}^{-},v)).
			\end{equation*}
			Noticing that for any $(u,v)\in \Omega^{-}(c_1,c_2)$, $t_{u,v}^{-}=1$  and $B\subset\Omega^{-}(c_1,c_2)$, we have $\alpha(t,u,v)=(u,v)$ for $(t,u,v)\in (\{0\}\times S(c_1)\times S(c_2)\cup([0,1]\times B)$. Since $\alpha$ is continuous,  by Definition \ref{GhoussobDefinition} we have 
			\begin{equation*}
				A_n:=\alpha(\{1\}\times D_n)=\left\{(\varphi(t_{u,v}^{-},u),\varphi(t_{u,v}^{-},v)):(u,v)\in D_n\right\}\in\mathcal{F}.
			\end{equation*}
			It is clear that $A_n\subset\Omega^{-}(c_1,c_2)$ for all $n\in\mathbb{N}$. Due to the definition of $\alpha$ and $I^{-}$, one has that $\max_{D_{n}}I^{-}=\max_{A_{n}}I^{-}$ and hence $\{A_n\}\subset\Omega^{-}(c_1,c_2)$ is another minimizing sequence of $e_{\mathcal{F}}^{-}$. Using Lemma \ref{Minimaxthm}, we obtain a Palais-Smale sequence $\{(\hat{u}_n,\hat{v}_n)\}$ for $I^{-}$ in $S(c_1)\times S(c_2)$ at level $e_{\mathcal{F}}^{-}$ such that $\text{dist}((\hat{u}_n,\hat{v}_n),A_n)\to 0$ as $n\to\infty$. Let  $t_n=t_{\hat{u}_n,\hat{v}_n}^{-}$ and $(u_n,v_n)=(\varphi(t_n,\hat{u}_n),\varphi(t_n,\hat{v}_n))\in\Omega^{-}(c_1,c_2)$. We claim that there exists $C>0$ such that 
			\begin{equation}\label{estimate1}
				\frac{1}{C}\leqslant t_{n}^{2}\leqslant C
			\end{equation}
			for $n\in\mathbb{N}$ large enough. Indeed, Note that 
			\begin{equation*}
				t_{n}^{2}=\frac{Q(u_n)+Q(v_n)}{Q(\hat{u}_{n})+Q(\hat{v}_{n})}.
			\end{equation*}
			By definition of $I^{-}$, one has $I(u_n,v_n)=I^{-}(\hat{u}_n,\hat{v}_n)$ and $I(u_n,v_n)\to e_{\mathcal{F}}^{-}$, as $n\to\infty$. Note that the functional $I$ restricted on $\Omega(c_1,c_2)$ can be represented as 
			\begin{equation*}
				I(u,v)=\frac{p-2}{2(p-1)}\left(Q(u)+Q(v)\right)+\frac{1}{4}W_0(u,v)+\frac{(c_1+c_2)^2}{8(p-1)}.
			\end{equation*}
			Then from Lemma \ref{lemmathree}(i), one has
			\begin{equation}\label{IgeqQ}
				e_{\mathcal{F}}^{-}+o(1)=I(u_n,v_n)\geqslant\frac{p-2}{2(p-1)}\left(Q(u_n)+Q(v_n)\right)-\frac{C}{4}\left(c_{1}^{\frac{3}{4}}Q(u_n)^{\frac{1}{4}}+c_{2}^{\frac{3}{4}}Q(v_n)^{\frac{1}{4}}\right)^2+\frac{(c_1+c_2)^2}{8(p-1)}.
			\end{equation}
			By \eqref{IgeqQ} and Lemma \ref{lemma4}, there exists $M_1>0$ such that 
			\begin{equation*}
				\frac{1}{M_1}\leqslant Q(u_n)+Q(v_n)\leqslant M_1.
			\end{equation*}
			Since $\{A_n\}\subset\Omega^{-}(c_1,c_2)$ is a minimizing sequence for $e_{\mathcal{F}}^{-}$, from \eqref{IgeqQ} we know that $\{A_n\}$ is uniformly bounded in $H^1(\R^2)\times H^1(\R^2)$. Then from $\text{dist}((\hat{u}_n,\hat{v}_n),A_n)\to 0$ as $n\to\infty$, there exists $M_2>0$ such that 
			\begin{equation*}
				\frac{1}{M_2}\leqslant Q(\hat{u}_{n})+Q(\hat{v}_{n})\leqslant M_2.
			\end{equation*}
			This proves the claim.
			
			We show that $(u_n,v_n)$ is the Palais-Smale sequence for $I$ restricted to $S(c_1)\times S(c_2)$ at level $e_{\mathcal{F}}^{-}$. Denoting by $\lVert\cdot\rVert_{*}$ the norm of $(T_{(u_n,v_n)}(S(c_1)\times S(c_2)))^*$, it holds
			\begin{equation*}
				\begin{aligned}
					&\lVert dI|_{S(c_1)\times S(c_2)}(u_n,v_n)\rVert_{*}=\sup_{\substack{(\phi,\psi)\in T_{(u_n,v_n)}(S(c_1)\times S(c_2)),\\\lVert (\phi,\psi)\rVert\leqslant 1}}\left|\langle dI|_{S(c_1)\times S(c_2)}(u_n,v_n),(\phi,\psi)\rangle\right|\\
					=&\sup_{\substack{(\phi,\psi)\in T_{(u_n,v_n)}(S(c_1)\times S(c_2)),\\\lVert (\phi,\psi)\rVert\leqslant 1}}\left|\langle dI|_{S(c_1)\times S(c_2)}(u_n,v_n),(\varphi(t_n,\varphi(\frac{1}{t_n},\phi)),\varphi(t_n,\varphi(\frac{1}{t_n},\psi)))\rangle\right|.
				\end{aligned}
			\end{equation*}
			From Lemma \ref{lemme2} we know that $T_{(\hat{u}_n,\hat{v}_n)}(S(c_1)\times S(c_2))\to T_{(u_n,v_n)}(S(c_1)\times S(c_2))$ defined by $(\phi,\psi)\mapsto(\varphi(t_n,\phi),\varphi(t_n,\psi))$ is isomorphism. Then it follows from Lemma \ref{lemme3} that 
			\begin{equation}\label{estimate2}
				\begin{aligned}
					\lVert dI|_{S(c_1)\times S(c_2)}(u_n,v_n)\rVert_{*}
					=\sup_{\substack{(\phi,\psi)\in T_{(u_n,v_n)}(S(c_1)\times S(c_2)),\\\lVert (\phi,\psi)\rVert\leqslant 1}}\left|\langle dI^{-}|_{S(c_1)\times S(c_2)}(\hat{u}_n,\hat{v}_n),(\varphi(\frac{1}{t_n},\phi),\varphi(\frac{1}{t_n},\psi))\rangle\right|.
				\end{aligned}
			\end{equation}
			By \eqref{estimate1} one has that (increasing $C$ if necessary) $\lVert(\varphi(\frac{1}{t_n},\phi),\varphi(\frac{1}{t_n},\psi))\rVert\leqslant C\lVert (\phi,\psi)\rVert\leqslant C$, hence it follows from \eqref{estimate2} that $(u_n,v_n)$ is the Palais-Smale sequence for $I$ restricted to $S(c_1)\times S(c_2)$ at level $e_{\mathcal{F}}^{-}$.
		\end{proof}
		
		\begin{lemma}\label{lemme5}
			There exists a Palais-Smale sequence $\left\{(u_n,v_n)\right\}\subset\Omega^{-}(c_1,c_2)$ for $I$ restricted to $S(c_1)\times S(c_2)$ at the level $m_2$.
		\end{lemma}
		\begin{proof}
			Take the class $\overline{\mathcal{F}}$ of all singletons belonging to $S(c_1)\times S(c_2)$ and $B=\varnothing$. By Definition \ref{GhoussobDefinition}, $\overline{\mathcal{F}}$ is a homotopy-stable family of compact subsets of $S(c_1)\times S(c_2)$ without boundary. Then by Lemma \ref{lemme4} we obtain the Palais-Smale sequence $\left\{(u_n,v_n)\right\}\subset\Omega^{-}(c_1,c_2)$ for $I$ restricted to $S(c_1)\times S(c_2)$ at level $e_{\mathcal{F}}^{-}$. Moreover, note that
			\begin{equation*}
				e_{\mathcal{F}}^{-}:=\inf_{A\in\overline{\mathcal{F}}}\max_{(u,v)\in A}I^{-}(u,v)=\inf_{(u,v)\in S(c_1)\times S(c_2)}I^{-}(u,v)=m_2,
			\end{equation*}
			which completes the proof.
		\end{proof}
		
		\begin{proof}[Proof of Theorem \ref{thm2}\rm{(ii)}]
			In virtue of Lemma \ref{lemme5}, we obtain the Palais-Smale sequence $\{(u_n,v_n)\}\subset\Omega^{-}(c_1,c_2)$ for $I$ restricted to $S(c_1)\times S(c_2)$ at level $m_2$. The definition of $\Omega^{-}(c_1,c_2)$ shows that  
			\begin{equation}\label{formula in step1}
				I(u_n,v_n)=\frac{p-2}{2(p-1)}\left(Q(u_n)+Q(v_n)\right)+\frac{1}{4}W_0(u_n,v_n)+\frac{(c_1+c_2)^2}{8(p-1)}=m_2+o(1),
			\end{equation}
			as $n\to\infty$. Then, from Lemma \ref{lemmathree}(i), we have
			\begin{equation*}
				\frac{p-2}{2(p-1)}\left(Q(u_n)+Q(v_n)\right)-\frac{C}{4}\left(c_{1}^{\frac{3}{4}}Q(u_n)^{\frac{1}{4}}+c_{2}^{\frac{3}{4}}Q(v_n)^{\frac{1}{4}}\right)^2+\frac{(c_1+c_2)^2}{8(p-1)}\leqslant m_2+o(1),
			\end{equation*}
			which implies that $\left\{(u_n,v_n)\right\}$ is bounded in $H^1(\R^2)\times H^1(\R^2)$. Therefore, from \eqref{formula in step1} we know that $\left\{W_1(u_n,v_n)\right\}$ is bounded. By Lemma \ref{weak convergent sequence}, there exists a sequence $\left\{x_n\right\}\subset\R^2$ and some $(u^{-},v^{-})\in S(c_1)\times S(c_2)$ such that, up to a subsequence, 
			\begin{equation}\label{wconvergencOmega-}
				(\tilde{u}_n,\tilde{v}_n)\rightharpoonup(u^{-},v^{-})\text{ in } S(c_1)\times S(c_2),\quad\text{as }n\to\infty,
			\end{equation}
			where $(\tilde{u}_n,\tilde{v}_n)=(u_n(\cdot-x_n),v_n(\cdot-x_n))$. It also holds that sequence $\left\{(\tilde{u}_n,\tilde{v}_n)\right\}\subset \Omega^{-}(c_1,c_2)$ is a Palais-Smale sequence for $I$ at level $m_2$. By Definition \ref{definition of critical point}, there exist two real number sequences $\left\{\lambda_{1,n}\right\}$ and $\left\{\lambda_{2,n}\right\}$ such that
			\begin{equation*}
				dI(\tilde{u}_n,\tilde{v}_n)+\lambda_{1,n}(\tilde{u}_n,0)+\lambda_{2,n}(0,\tilde{v}_n)\to 0\text{ in }E^{*}\times E^{*},\quad\text{as }n\to\infty.
			\end{equation*}
			Hence,
			\begin{equation}\label{boundlambda1}
				\begin{aligned}
					o(1)&=\langle dI(\tilde{u}_n,\tilde{v}_n),(\tilde{u}_n,0)\rangle+\lambda_{1,n}c_1\\
					&=Q(\tilde{u}_n)+V_1(\tilde{u}_{n})+B_1(\tilde{v}_{n}^2,\tilde{u}_{n}^2)-B_2(\tilde{u}_{n}^2+\tilde{v}_{n}^2,\tilde{u}_{n}^2)-\mu_{1}P(\tilde{u}_n)-\beta P_0(\tilde{u}_n,\tilde{v}_n)+\lambda_{1,n}c_1,
				\end{aligned}
			\end{equation}
			\begin{equation}\label{boundlambda2}
				\begin{aligned}
					o(1)&=\langle dI(\tilde{u}_n,\tilde{v}_n),(0,\tilde{v}_n)\rangle+\lambda_{2,n}c_2\\
					&=Q(\tilde{v}_n)+V_1(\tilde{v}_{n})+B_1(\tilde{u}_{n}^2,\tilde{v}_{n}^2)-B_2(\tilde{u}_{n}^2+\tilde{v}_{n}^2,\tilde{v}_{n}^2)-\mu_{2}P(\tilde{v}_n)-\beta P_0(\tilde{u}_n,\tilde{v}_n)+\lambda_{2,n}c_2.
				\end{aligned}
			\end{equation}
			Due to the invariance by translation, the sequence $\left\{(\tilde{u}_n,\tilde{v}_n)\right\}$ is bounded in $H^1(\R^2)\times H^1(\R^2)$ and $\left\{W_1(\tilde{u}_n,\tilde{v}_n)\right\}$ is bounded.  Using Lemma \ref{lemmathree}, \eqref{B2} and \eqref{decomposition} in \eqref{boundlambda1} and \eqref{boundlambda2}, it follows that $\left\{\lambda_{1,n}\right\}$ and $\left\{\lambda_{2,n}\right\}$ are both bounded. Passing to subsequence, we assume that $\lambda_{1,n}\to\lambda_{1}^{-}$ and $\lambda_{2,n}\to\lambda_{2}^{-}$, as $n\to\infty$. Then we obtain
			\begin{equation*}
				dI(\tilde{u}_n,\tilde{v}_n)+\lambda_{1}^{-}(\tilde{u}_n,0)+\lambda_{2}^{-}(0,\tilde{v}_n)\to 0\text{ in }E^{*}\times E^{*},\quad\text{as }n\to\infty.
			\end{equation*}
			Since $(u,v)\mapsto dI(u,v)+\lambda_{1}^{-}(u,0)+\lambda_{2}^{-}(0,v)$ is weak-to-weak* continuous, it follows from \eqref{wconvergencOmega-} that $(\lambda_{1}^{-},\lambda_{2}^{-},u^{-},v^{-})$ is actually a solution of problem \eqref{2-NHF} and $I(u^{-},v^{-})=m_2$. By Lemma \ref{lemma2} and weak lower semicontinuity of $Q$, we have $\M(u^{-},v^{-})=0$ and $g_{u^{-},v^{-}}(1)<0$, which implies by \eqref{Omega} that $(u^{-},v^{-})\in\Omega^{-}(c_1,c_2)$, where $g_{u^{-},v^{-}}$ comes from \eqref{guv(t)}.
			Thus,  we know from Lemma \ref{lemma5} that $I(u^{+},v^{+})=m_1<m_2=I(u^{-},v^{-})$. 
		\end{proof}
		
		\section*{Appendix A: Proof of Lemma \ref{Pohožaev-Nehari identity} }
		\begin{appendix}
			\setcounter{equation}{0}
			\renewcommand\theequation{A.\arabic{equation}}
			\begin{proof}
				Using the similar argument in \cite[Lemma 2.4]{CFMM22}, we know that $u,v\in C^2(\R^2)$. Moreover, $w_{u,v}:=\int_{\R^2}\log |x-y|\left(u^2(y)+v^2(y)\right)dy$ is also of class $C^2(\R^2)$. Inspired by \cite[Proposition 1]{BL1983}, we multiply the first equation by $x\cdot\nabla u$ and the second equation by $x\cdot\nabla v$. Integrating on the ball $B_R(0)$ for $R>0$ and adding two equations, we have
				\begin{equation}\label{sum}
					\begin{aligned}
						&\int_{B_R(0)}\left[-\Delta u(x\cdot\nabla u)+\lambda_1u(x\cdot\nabla u)+w_{u,v}u(x\cdot\nabla u)\right]dx\\	&\quad+\int_{B_R(0)}\left[-\Delta v(x\cdot\nabla v)+\lambda_2v(x\cdot\nabla v)+w_{u,v}v(x\cdot\nabla v)\right]dx\\
						&=\int_{B_R(0)}\big[\mu_{1}\left|u\right|^{2p-2}u(x\cdot\nabla u)+\mu_{2}\left|v\right|^{2p-2}v(x\cdot\nabla v)\\
						&\quad\quad+\beta\left|v\right|^{p}\left|u\right|^{p-2}u(x\cdot\nabla u)+\beta\left|u\right|^{p}\left|v\right|^{p-2}v(x\cdot\nabla v)\big]dx.
					\end{aligned}
				\end{equation}
				For any $u\in C^2(\R^2)$,  it satisfies
				\begin{equation*}
					\Delta u(x\cdot\nabla u)=\mathrm{div}\Big(\nabla u(x\cdot\nabla u)-x\frac{|\nabla u|^{2}}{2}\Big)\quad\mathrm{on}~ \mathbb{R}^{2},
				\end{equation*}
				so using the divergence theorem, one has
				\begin{equation}\label{sum1}
					\int_{B_{R}(0)}-\Delta u(x\cdot\nabla u)dx=-\frac{1}{R}\int_{\partial B_{R}(0)}|x\cdot\nabla u|^{2}dS+\frac{R}{2}\int_{\partial B_{R}(0)}|\nabla u|^{2}dS.
				\end{equation}
				Then we have
				\begin{equation}\label{sum2}
					\int_{B_R(0)}\lambda_1u(x\cdot\nabla u)dx=R\int_{\partial B_{R}(0)}\dfrac{\lambda_{1}}{2}|u|^2dS-\int_{B_{R}(0)}\lambda_{1}|u|^2dx,
				\end{equation}
				and
				\begin{equation}\label{sum3}
					\int_{B_R(0)}\mu_{1}\left|u\right|^{2p-2}u(x\cdot\nabla u)dx=R\int_{\partial B_{R}(0)}\dfrac{\mu_{1}}{2p}|u|^{2p}dS-\int_{B_{R}(0)}\dfrac{\mu_{1}}{p}|u|^{2p}dx.
				\end{equation}
				Note that 
				\begin{equation*}
					w_{u,v}u(x\cdot\nabla u)=\frac{1}{2}\big(\mathrm{div}[xw_{u,v}u^2]-u^2(x\cdot\nabla w_{u,v})-2w_{u,v}u^2\big).
				\end{equation*}
				In view of the divergence theorem, we have
				\begin{equation}\label{sum4}
					\int_{B_R(0)}w_{u,v}u(x\cdot\nabla u)dx=\frac{R}{2}\int_{\partial B_R(0)}w_{u,v}u^2dS-\frac{1}{2}\int_{B_R(0)}u^2(x\cdot\nabla w_{u,v})dx-\int_{B_R(0)}w_{u,v}u^2dx.
				\end{equation}
				Similarly, it holds that
				\begin{equation}\label{sum5}
					\int_{B_{R}(0)}-\Delta v(x\cdot\nabla v)dx=-\frac{1}{R}\int_{\partial B_{R}(0)}|x\cdot\nabla v|^{2}dS+\frac{R}{2}\int_{\partial B_{R}(0)}|\nabla v|^{2}dS,
				\end{equation}
				\begin{equation}\label{sum6}
					\int_{B_R(0)}\lambda_2v(x\cdot\nabla v)dx=R\int_{\partial B_{R}(0)}\dfrac{\lambda_{2}}{2}|v|^2dS-\int_{B_{R}(0)}\lambda_{2}|v|^2dx,
				\end{equation}
				\begin{equation}\label{sum7}
					\int_{B_R(0)}\mu_{2}\left|v\right|^{2p-2}v(x\cdot\nabla v)dx=R\int_{\partial B_{R}(0)}\dfrac{\mu_{2}}{2p}|v|^{2p}dS-\int_{B_{R}(0)}\dfrac{\mu_{2}}{p}|v|^{2p}dx,
				\end{equation}
				and
				\begin{equation}\label{sum8}
					\int_{B_R(0)}w_{u,v}v(x\cdot\nabla v)dx=\frac{R}{2}\int_{\partial B_R(0)}w_{u,v}v^2dS-\frac{1}{2}\int_{B_R(0)}v^2(x\cdot\nabla w_{u,v})dx-\int_{B_R(0)}w_{u,v}v^2dx.
				\end{equation}
				In addition, one has
				\begin{equation}\label{sum9}
					\begin{aligned}	\int_{B_R(0)}\left[\beta\left|v\right|^{p}\left|u\right|^{p-2}u(x\cdot\nabla u)+\beta\left|u\right|^{p}\left|v\right|^{p-2}v(x\cdot\nabla v)\right] dx
						=R\int_{\partial B_R(0)}\dfrac{\beta}{p}|uv|^pdS-\int_{B_R(0)}\dfrac{2\beta}{p}|uv|^pdx.
					\end{aligned}
				\end{equation}
				According to \eqref{sum1}-\eqref{sum9}, \eqref{sum} can be rewritten as
				\begin{equation}\label{boundary}
					\begin{aligned}
						&\int_{B_R(0)}\Big[-\lambda_1|u|^2-\lambda_2|v|^2+\dfrac{\mu_1}{p}|u|^{2p}+\dfrac{\mu_2}{p}|v|^{2p}\\
						&\quad\quad\quad-\dfrac{1}{2}\left(u^2+v^2\right)(x\cdot\nabla w_{u,v})-w_{u,v}\left(u^2+v^2\right)+\dfrac{2\beta}{p}|uv|^p\Big]dx\\
						&=\int_{\partial B_R(0)}\Big[\dfrac{|x\cdot \nabla u|^2}{R}+\dfrac{|x\cdot \nabla v|^2}{R}+R\Big(\dfrac{\beta}{p}|uv|^p-\dfrac{|\nabla u|^2}{2}-\dfrac{|\nabla v|^2}{2}-\dfrac{\lambda_1}{2}|u|^2-\dfrac{\lambda_2}{2}|v|^2\\
						&\quad\quad\quad\quad\quad+\dfrac{\mu_1}{2p}|u|^{2p}+\dfrac{\mu_2}{2p}|v|^{2p}-\dfrac{1}{2}w_{u,v}(u^2+v^2)\Big)\Big]dS.
					\end{aligned}
				\end{equation}
				We define function $h:\R^2\to\R$ by
				\begin{equation*}
					\begin{aligned}
						h(x):=&\dfrac{|x\cdot \nabla u|^2}{|x|^2}+\dfrac{|x\cdot \nabla v|^2}{|x|^2}+\dfrac{\beta}{p}|uv|^p-\dfrac{|\nabla u|^2}{2}-\dfrac{|\nabla v|^2}{2}\\
						&\quad-\dfrac{\lambda_1}{2}|u|^2-\dfrac{\lambda_2}{2}|v|^2
						+\dfrac{\mu_1}{2p}|u|^{2p}+\dfrac{\mu_2}{2p}|v|^{2p}-\dfrac{1}{2}w_{u,v}(u^2+v^2).
					\end{aligned}
				\end{equation*}
				Since $(u,v)\in E\times E\subset H^1(\R^2)\times H^1(\R^2)$, it is easy to check that $h\in L^1(\R^2)$. Now we claim that there exists a sequence $R_n\rightarrow+\infty$ such that 
				\begin{equation*}
					R_n\int_{\partial B_{R_n}(0)}|h|dS\rightarrow0.
				\end{equation*}
				Otherwise, there exist constants $R_0>0$ and $\varrho>0$ such that 
				\begin{equation*}
					R\int_{\partial B_{R}(0)}|h|dS\geq \varrho>0,\quad\mathrm{for}~R\geq R_0.
				\end{equation*}
				Hence,
				\begin{equation*}
					\int_{\R^2}|h|dx=\int_{0}^{+\infty}\left(R\int_{\partial B_R(0)}|h|dS\right)dR\geq\int_{R_0}^{+\infty}\varrho dR=\infty,
				\end{equation*}
				which is a contradiction to $h\in L^1(\R^2)$. 
				
				Letting $R_n\rightarrow+\infty$ in \eqref{boundary}, we obtain 
				\begin{equation}\label{R2}
					\begin{aligned} 
						&\int_{\R^2}\Big[-\lambda_1|u|^2-\lambda_2|v|^2+\dfrac{\mu_1}{p}|u|^{2p}+\dfrac{\mu_2}{p}|v|^{2p}\\
						&\quad\quad\quad-\dfrac{1}{2}\left(u^2+v^2\right)(x\cdot\nabla w_{u,v})-w_{u,v}\left(u^2+v^2\right)+\dfrac{2\beta}{p}|uv|^p\Big]dx=0.
					\end{aligned}
				\end{equation}
				Then we know that $\left(u^2+v^2\right)(x\cdot\nabla w_{u,v})\in L^1(\R^2)$. Note that
				\begin{equation*}
					x\cdot\nabla w_{u,v}=\int_{\R^2}\dfrac{|x|^2-x\cdot y}{|x-y|^2}\left(u^2(y)+v^2(y)\right)dy.
				\end{equation*}
				So we have
				\begin{equation*}
					\begin{aligned}
						\int_{\R^2}(u^2+v^2)(x\cdot\nabla w_{u,v})&=\iint_{\R^2\times\R^2}\dfrac{|x|^2-x\cdot y}{|x-y|^2}\left(u^2(x)+v^2(x)\right)\left(u^2(y)+v^2(y)\right)dxdy\\
						&=\dfrac{1}{2}\iint_{\R^2\times\R^2}\dfrac{|x|^2+|y|^2-2 x\cdot y}{|x-y|^2}\left(u^2(x)+v^2(x)\right)\left(u^2(y)+v^2(y)\right)dxdy\\
						&=\dfrac{1}{2}\left(\int_{\R^2}(|u|^2+|v|^2)dx\right)^2.
					\end{aligned}
				\end{equation*}
				From \eqref{R2}, we obtain the following  Poho\v{z}aev identity 
				\begin{equation}\label{Pohožaev identity}
					\begin{aligned}
						\lambda_1&\int_{\R^2}|u|^2dx+\lambda_2\int_{\R^2}|v|^2dx+\iint_{\R^2\times\R^2}\log|x-y|\left(u^2(x)+v^2(x)\right)\left(u^2(y)+v^2(y)\right)dxdy\\
						&+\dfrac{1}{4}\left(\int_{\R^2}(|u|^2+|v|^2)dx\right)^2=\dfrac{1}{p}\left(\mu_1\int_{\R^2}|u|^{2p}dx+\mu_2\int_{\R^2}|v|^{2p}dx+2\beta\int_{\R^2}|uv|^pdx\right).
					\end{aligned}
				\end{equation}
				Multiplying the first equation of (\ref{equations of SP}) by $u$  and integrating over $\R^2$, we have
				\begin{equation}\label{Nehari1}
					\begin{aligned}
						\int_{\R^2}|\nabla u|^2dx&+\lambda_1\int_{\R^2}|u|^2dx+\iint_{\R^2\times\R^2}\log|x-y|\left(u^2(y)+v^2(y)\right)u^2(x)dxdy\\
						&=\mu_1\int_{\R^2}|u|^{2p}dx+\beta\int_{\R^2}|uv|^{p}dx.
					\end{aligned}
				\end{equation}
				Similarly, we also have
				\begin{equation}\label{Nehari2}
					\begin{aligned}
						\int_{\R^2}|\nabla v|^2dx&+\lambda_2\int_{\R^2}|v|^2dx+\iint_{\R^2\times\R^2}\log|x-y|\left(u^2(y)+v^2(y)\right)v^2(x)dxdy\\
						&=\mu_2\int_{\R^2}|v|^{2p}dx+\beta\int_{\R^2}|uv|^{p}dx.
					\end{aligned}
				\end{equation}
				Then from (\ref{Nehari1}) and (\ref{Nehari2}), we have 
				\begin{equation}\label{Nehari}
					\begin{aligned}
						&\int_{\R^2}|\nabla u|^2dx+\int_{\R^2}|\nabla v|^2dx+\lambda_1\int_{\R^2}|u|^2dx+\lambda_2\int_{\R^2}|v|^2dx\\
						&+\iint_{\R^2\times\R^2}\log|x-y|\left(u^2(x)+v^2(x)\right)\left(u^2(y)+v^2(y)\right)dxdy\\        &=\mu_1\int_{\R^2}|u|^{2p}dx+\mu_2\int_{\R^2}|v|^{2p}dx+2\beta\int_{\R^2}|uv|^{p}dx.
					\end{aligned}
				\end{equation}
				Combining (\ref{Pohožaev identity}) and (\ref{Nehari}), we finally obtain $\M(u,v)=0$. Then the proof is completed.
			\end{proof}
		\end{appendix}
		
		\noindent{\bf Acknowledgements.}
		This work is supported by the National Natural Science Foundation of China (Grant No. 12571120, 12271508),  and Scientific Research Project of Education Department of Jilin Province (Grant No. JJKH20220964KJ). X. Zeng is partly
		supported by the National Science Centre, Poland (Grant No. 2023/51/B/ST1/00968).
		
		\noindent{\bf Data availability.}	The manuscript has no associated data.

		\noindent{\bf Conflict of interest.}	The authors declare that they have no conflict of interest.

		\bibliographystyle{abbrv}
		\bibliography{ref}
		
	\end{document}